\documentclass[reqno]{amsart}
\usepackage[utf8]{inputenc}  
\usepackage{graphicx}
\usepackage{pgfplots}
\usepackage{float}
\usepackage{bm}
\usepackage{stmaryrd}
\usepackage{subfig}
\usepackage{graphicx}
\usepackage{xcolor}

\usepackage{booktabs}
\usepackage{adjustbox}
\usepackage{siunitx}

\usepackage{algorithm}
\usepackage{algpseudocode}

\usepackage{amsmath,amssymb, mathtools}
\usepackage{multicol}
\usepackage{hyperref}
\usepackage{cleveref}
\usepackage{dsfont}

\usepackage[a4paper,top=3 cm,bottom=3 cm,left=2.5 cm,right=2.5 cm]{geometry}
\newtheorem{thrm}{Theorem}[section]
\newtheorem{lmm}[thrm]{Lemma}

\newtheorem{prpstn}[thrm]{Proposition}

\newtheorem{crllr}[thrm]{Corollary}

\allowdisplaybreaks
\numberwithin{equation}{section}

\def\R{\mathbb R}

\def\to{\rightarrow}

\DeclareMathOperator*{\dive}{div}

\DeclareMathOperator*{\dom}{dom}

\DeclareMathOperator*{\argmin}{argmin}
\DeclareMathOperator*{\Var}{Var}
\DeclareMathOperator*{\sgn}{sgn}
\DeclareMathOperator*{\proj}{proj}
\DeclareMathOperator*{\inte}{int}

\usepackage{titletoc}

\makeatletter
\newcommand\thankssymb[1]{\textsuperscript{\@fnsymbol{#1}}}

\makeatother

\begin{document}

\title[Mini-batch descent in semiflows]{Mini-batch descent in semiflows}

\author[A. Dom\'{i}nguez]{
Alberto Dom\'{i}nguez Corella\thankssymb{2}}
\email{alberto.corella@fau.de}

\author[M. Hern\'{a}ndez]{Mart\' in Hern\' andez\thankssymb{2}}
\email{martin.hernandez@fau.de}

\thanks{\thankssymb{2}
 Chair for Dynamics, Control, Machine Learning, and Numerics, Alexander von Humboldt-Professorship, Department of Mathematics,  Friedrich-Alexander-Universit\"at Erlangen-N\"urnberg,
91058 Erlangen, Germany.}


\begin{abstract}
This paper investigates the application of mini-batch gradient descent to semiflows (gradient flows). Given a loss function (potential), we introduce a continuous version of mini-batch gradient descent by randomly selecting sub-loss functions over time, defining a piecewise flow. We prove that, under suitable assumptions on the potential generating the semiflow, the \textit{mini-batch descent flow} trajectory closely approximates the original semiflow trajectory on average. In addition, we study a randomized minimizing movement scheme that also approximates the semiflow of the full loss function. We illustrate the versatility of this approach across various problems, including constrained optimization, sparse inversion, and domain decomposition. Finally, we validate our results with several numerical examples.
\end{abstract}
\maketitle

\section{Introduction}
Probability theory has significantly impacted various areas of mathematics, particularly algorithms and combinatorics. The use of randomness in algorithms dates back to Monte Carlo methods \cite{Metropolis1953,Ulam1949}. In subsequent years, random algorithms gained strength in optimization, notably with the emergence of techniques such as simulated annealing \cite{Kirkpatrick1983} and genetic algorithms \cite{Holland1975}, which were used in complex optimization problems. Recently, stochastic gradient descent algorithms have attracted attention in artificial intelligence; their development has been crucial in training large-scale models, especially machine learning algorithms \cite{Bottou2010,Bottou2018}.
\smallbreak
A notable variant of stochastic descent algorithms is the so-called mini-batch gradient descent. Unlike vanilla stochastic gradient descent, mini-batch implementations typically aggregate gradients over small batches, reducing gradient variance \cite{Qian2020impact}. This variant balances robustness and efficiency and has become one of the most widely used implementations of gradient descent in deep learning.
The continuous equivalent of gradient descent is the so-called gradient flow. With small learning rates, mini-batch gradient descent mimics the gradient flow trajectory of the full batch loss function. This paper is concerned with mini-batch descent applied to \textit{semiflows }(gradient flows of convex non-necessarily differentiable functions). Given a loss function expressed as the average of \textit{sub-loss} functions, we compare the gradient flow of the full loss function with a flow that does not strictly follow the steepest descent of the full loss function but instead follows portions of it, determined by randomly chosen batches of sub-loss functions changing over time. We refer to this process as \textit{mini-batch descent flow}.
A full description of how these trajectories are generated is given in the next section.
\smallbreak
Our main interest lies in comparing trajectories rather than finding minima of the full loss function through the mini-batch descent-generated trajectory. Provided that the replacement of batches is done sufficiently often, we prove that the trajectory generated by the mini-batch descent flow is close (in expectation) to the original trajectory generated by the gradient flow of the full loss function.  By examining the paths taken by the trajectories of the mini-batch descent flow, it is possible to gain insights into their convergence properties and robustness to noise.
\smallbreak
For simplicity, we focus on proper lower semicontinuous convex functionals defined on Hilbert spaces. With some standard modifications, the results also hold for semi-convex functionals by replacing the convex subdifferential with the limiting one. An advantage of considering gradient flows in Hilbert spaces is that this framework allows several solutions of evolution equations to be interpreted as gradient flows of suitable functionals.
To address non-single-valued subdifferentials, we introduce an appropriate concept of a non-biased estimator, ensuring that the expected value of the randomly chosen subgradient aligns with the minimal norm subdifferential of the loss function at each step. We quantify the discrepancy of the mini-batch descent using a measure function analogous to the one used in stochastic gradient-based methods to quantify algorithm variance.
\smallbreak
Additionally, we introduce a suitable randomized version of the \textit{minimizing movement scheme} and study its relationship with the gradient flow of the full loss function. Through rigorous analysis, we provide convergence estimates and conditions under which this randomized minimizing movement scheme can approximate the gradient flow. This has significant implications for the long-term behavior of the trajectories generated by this mini-batch descent, demonstrating that they can reach arbitrarily good approximate minimizers of the full loss function.
\smallbreak
A full account of the results is given in the next section. Let us now provide an overview of the related literature and highlight our contributions.
\smallbreak
The idea of comparing the dynamics generated by an equation with a dynamic changing over time (as in stochastic algorithms) was first employed in \cite{ShietalRBM} for solving particle interacting dynamics under the name of random batch methods. These methods have been further extended to address other problems \cite{ShietalRBM2,RBM1,RBM2,10463131}. The mini-batch descent scheme we use is based on \cite[Algorithm 2b]{ShietalRBM}. This algorithm was further used in \cite{LQRVeldman} to generate approximate optimal control trajectories. 
We extend the convergence results for state trajectories previously obtained in \cite{LQRVeldman} from positive definite matrices to general convex functions and from Euclidean to Hilbert spaces. Infinite-dimensional spaces offer several advantages for problems where the \textit{optimize then discretize} paradigm is more appropriate. Results for optimal control problems with gradient flow dynamics can potentially be derived by adjusting the proofs in \cite{LQRVeldman} with the convergence results presented here. Additionally, by allowing set-valued operators from a subdifferential, further results for control-constrained problems can be considered.
We give an example of how trajectories look when constraints are induced by means of indicator functions in Section \ref{projdyn}.
\smallbreak
In \cite{Latz}, a stochastic process was introduced as a continuous-time representation of the stochastic gradient descent algorithm. This process is characterized as a dynamical system coupled with a continuous-time index process living on a finite state space, wherein the dynamical system, representing the gradient descent part, is coupled with the process on the finite state space, representing the random sub-sampling. It is proved there that the process converges weakly to the gradient flow with respect to the full target function as the step size approaches zero. In contrast, the process studied in this paper does not randomly select when to change from following the steepest descent of a sub-loss function to follow another, i.e., the waiting time is fixed and deterministic. However, we obtain much more robust estimates for the trajectories. We also mention \cite{Latz2}, where sparse inversion and classification problems were studied from this same continuous-time perspective. Inspired by this, we also consider a sparse inversion problem in Section \ref{Sparseinv} for illustrative purposes of how our abstract results can be applied.

\smallbreak

Let us now mention \cite{Eisenmann2024randomized}, where an abstract framework for solving evolutionary equations is considered. This framework is also inspired by stochastic algorithms used in the machine learning community. The main idea there is to consider a splitting of an operator and then construct a sequence that solves an implicit Euler scheme that at each step randomly selects operators from the splitting. This is in the same spirit as in \cite{ShietalRBM,LQRVeldman,10463131}. Our results for the minimizing movements are inspired by \cite{Eisenmann2024randomized} and are formulated in a somewhat different functional setting. Though the framework in \cite{Eisenmann2024randomized} is very general, it is also very abstract and requires a lot of verification even for simple problems. In comparison, our framework does not require much verification of spaces and operators; one just needs to provide the potential/loss function.
Also, gradient flows are very flexible, as, e.g., the heat flow can be seen as an $L^2$ flow or as an $H^{-1}$ flow. Moreover, while we do require operators to be monotone, we allow them to be non-single-valued. Thus, different problems can be considered, such as the porous media equation and the one-phase Stefan problem. One of the seminal contributions of \cite{Eisenmann2024randomized} is that they identify that mini-batch descent trajectories can be used for random domain decomposition algorithms to solve evolutionary equations; they present as an example a parabolic equation with the $p$-Laplacian. We present an example of domain decomposition for a parabolic obstacle problem, which comes from an operator that is non-single-valued.
\smallbreak

The paper is organized as follows. Section \ref{sec:problem_formulation} contains an overview of the problem and the main results. The proofs are carried out in Section \ref{sec:well_posedness_convergence}. In Section \ref{sec:applications}, we show how the proposed randomized schemes can be applied in different instances: constrained optimization, sparse inversion, and domain decomposition. Finally, in the Appendix, we provide further details on the decomposition of the minimal norm subgradient and the proposed variance measure based on it.
\section{Formulation of the Problem and Main Results}\label{sec:problem_formulation}

Let $\mathcal{H}$ be a Hilbert space and $\Phi:\mathcal{H} \to \mathbb{R} \cup \{+\infty\}$ a convex, lower semicontinuous, and proper function. For an initial datum $u_0 \in \dom\partial \Phi$, we consider the gradient flow equation
\begin{align}\label{gf_eq}
    \left\{ 
    \begin{array}{l} 
        \dot{u}(t) \in -\partial \Phi(u(t)) \quad t \in (0, \infty), \\  
        u(0) = u_0, 
    \end{array} 
    \right.
\end{align}
where $u_0$ is the initial data. A locally absolutely continuous function $u: [0, +\infty) \to \mathcal{H}$ is said to be a strong solution of (\ref{gf_eq}) if $u(0) = u_0$ and $0 \in \dot{u}(t) + \partial \Phi(u(t))$ for almost every $t \in [0, \infty)$. When $u_0 \in \dom\partial \Phi$, equation (\ref{gf_eq}) has a unique strong solution \cite[Theorem 17.2.2]{SoBV}.

\noindent
Let $p_1, \dots, p_n$ be positive numbers such that $\sum_{i=1}^n p_i = 1$. We assume that $\Phi$ can be represented as the average
\begin{align}\label{Vertretung}
    \Phi = \sum_{i=1}^{n} p_i \Phi_{i},
\end{align}
where $\Phi_{1}, \dots, \Phi_{n}: \mathcal{H} \to \mathbb{R} \cup \{+\infty\}$ are convex, lower semicontinuous, and proper functions satisfying 
\begin{align}\label{assu:contention_domains}
    \inf_{u \in \mathcal{H}} \Phi_{i}(u) > -\infty \qquad \text{and} \qquad \dom\Phi \subset \dom\Phi_{i} \subset \overline{\dom\Phi}, \quad \forall i \in \{1, \dots, n\}.
\end{align}
The idea behind representation (\ref{Vertretung}) is that the potential $\Phi$ can be obtained as the expected value of a random variable that takes the value $\Phi_i$ with probability $p_i$ for each $i \in \{1, \dots, n\}$.

\smallbreak
Let $\mathcal{B}_1, \dots, \mathcal{B}_m$ be nonempty sets such that $\bigcup_{j=1}^m \mathcal{B}_j = \{1, \dots, n\}$; these are called batches. To each batch $\mathcal{B}_j$, there is an associated positive probability $\pi_{j}$ in such a way that 
\[
p_i = \sum_{j: i \in \mathcal{B}_j} \frac{\pi_j}{|\mathcal{B}_j|} \quad \text{for each} \quad i \in \{1, \dots, n\}.
\]
The idea behind mini-batch descent is that the flow does not follow strictly the steepest descent of the full potential, but a portion of it determined by a randomly chosen batch.

For a set $\mathcal{B} \subset \{1, \dots, n\}$, let $\Phi_{\mathcal{B}}: \mathcal{H} \to \mathbb{R} \cup \{+\infty\}$ be given by
\begin{align*}
    \Phi_{\mathcal{B}}(u) := \frac{1}{|\mathcal{B}|} \sum_{i \in \mathcal{B}} \Phi_{i}(u).
\end{align*}

In order to describe the process formally, consider a positive parameter $\varepsilon > 0$ and a sequence $\{j_{k}\}_{k \in \mathbb{N}}$ of independent random variables. The positive parameter represents the waiting time at which the flow changes from following one subgradient batch to another at each step. On the other hand, each $j_k$ represents the choice of batch $\mathcal{B}_{j_k}$ at the $k^{\text{th}}$ step. The sequence $\{j_{k}\}_{k \in \mathbb{N}}$ is assumed to satisfy $\mathbb{P}(j_k = j) = \pi_j$ for each $j \in \{1, \dots, m\}$ and $k \in \mathbb{N}$. Set $t_{0} := 0$ and define recursively $t_{k} := t_{k-1} + \varepsilon$ for $k \in \mathbb{N}$.

\smallbreak
We will say that a continuous function $v_{\varepsilon}: [0, +\infty) \to \mathcal{H}$ is a solution of the \textit{mini-batch descent flow} scheme if $v_{\varepsilon}(0) = u_0$, $v_{\varepsilon}$ is locally absolutely continuous in $(t_{k-1}, t_k)$, and 
\begin{align}\label{eq:piecewise_flow}
    0 \in \dot{v}_{\varepsilon} + \partial \Phi_{\mathcal{B}_{j_k}}(v_{\varepsilon}) \quad \text{for almost every} \quad t \in [t_{k-1}, t_{k}) \quad \text{and for all} \quad k \in \mathbb{N}.
\end{align}
Before proceeding further, let us first address the well-posedness of this scheme, i.e., ensuring existence and uniqueness for the piecewise flow \eqref{eq:piecewise_flow}.

\begin{thrm}\label{existencethrm}
    There exists a unique solution $v_{\varepsilon}: [0, +\infty) \to \mathcal{H}$ of the \textit{mini-batch descent flow} scheme. Moreover, 
    \begin{itemize}
        \item[(i)] If $\dom \partial\Phi_{\mathcal{B}_j} \subset \dom\Phi$ for all $j \in \{1, \dots, m\}$, then $v_{\varepsilon}: [0, +\infty) \to \mathcal{H}$ is locally absolutely continuous;
        \item[(ii)] If $\dom \partial\Phi_{\mathcal{B}_j} \subset \dom\partial\Phi$ for all $j \in \{1, \dots, m\}$, then $v_{\varepsilon}: [0, +\infty) \to \mathcal{H}$ is locally Lipschitz continuous.
    \end{itemize}
\end{thrm}
\begin{proof}
    This follows from Theorem \ref{Existence} and Proposition \ref{Regularity}.
\end{proof}

To illustrate our mini-batch gradient descent method, \Cref{alg:mini-batch-descent} presents a step-by-step description of the dynamics construction. For the sake of clarity, we present the algorithm in a bounded time interval $(0,T)$. However, it can be simply modified to the interval $(0,\infty)$ by choosing $K=\infty$.

\begin{algorithm}[H]
\caption{mini-batch descent flow (finite time)}\label{alg:mini-batch-descent}
\begin{algorithmic}[1]
\Require Initial data $u_0$, step size $\varepsilon$, final time $T$, batches $\{\mathcal{B}_1, \dots, \mathcal{B}_m\}$, probabilities $\{\pi_1, \dots, \pi_m\}$

\State Initialize $v_{\varepsilon}(0) \gets u_0$

\State Define  $K\gets \lfloor T/\varepsilon \rfloor$

\For{$k \gets 1$ to $K$}
    \State Sample a batch index $j_k$ according to probabilities $\{\pi_1, \dots, \pi_m\}$
    \State Let $\mathcal{B} \gets \mathcal{B}_{j_k}$
    \State Compute the subgradient $\partial \Phi_{\mathcal{B}}(u) \gets  \partial\left(\frac{1}{|B|}\sum_{i\in B} \Phi_{i}(u)\right)$
    \State Solve $ \dot{v}_{\varepsilon} \in - \partial \Phi_{\mathcal{B}}(v_{\varepsilon})$ in $[t_{k-1},t_k)$
\EndFor

\Return $v_{\varepsilon}$
\end{algorithmic}
\end{algorithm}
We mention that a uniform grid was chosen for notational convenience; the arguments
of this article should still work if an unstructured time mesh we to be used.

	\subsection{Convergence}
	For each $t\in[0,+\infty)$, define $k_t:=\min\{k\in\mathbb N:\,\, t<t_{k}\}$.
	We can rewrite the mini-batch descent flow as 
	\begin{align}\label{mb_gf}
		\left\{ \begin{array}{l} \dot v_{\varepsilon}(t)  \in\displaystyle-\partial\Phi_{\mathcal B_{j_{k_t}}}(v_{\varepsilon}(t))\quad t\in(0,\infty),\\  v_{\varepsilon}(0)=u_0. \end{array} \right.
	\end{align}
	In order for the mini-batch descent to be a non-biased estimator, in the sense that the expected value of the randomly chosen subgradient coincides with the subdifferential of $\Phi$ at each step, the sum rule for subdifferentials  is assumed to hold, i.e., 
	\begin{align}\label{assu:sum_rule}
		\partial \Phi(u) = \sum_{j=1}^m \pi_{j}\, \partial \Phi_{\mathcal B_{j}}(u)\quad \text{for all $u\in\dom\partial\Phi$.}
	\end{align}
	This ensures that for each $j\in\{1,\dots,m\}$ there exists $\xi_j:\dom \partial \Phi\to \dom\partial \Phi_{\mathcal B_{j}}$ such that
	\begin{align}\label{eq:represent_subdif}
		\partial \Phi(u)^\circ = \sum_{j=1}^{m} \pi_j\xi_j(u)\quad\forall u\in\dom\partial\Phi,
	\end{align}
	where $\partial \Phi(u)^\circ = \argmin\{\big\|\xi\big\|_{\mathcal H}:\, \xi\in\partial\Phi(u)\}$ for all $u\in\dom\partial\Phi$. We see then that for each $k\in\mathbb N$,
	\begin{align*}
		\mathbb E \xi_{j_k}(u) = \partial \Phi(u)^\circ\quad \forall u\in\dom\partial \Phi.
	\end{align*}
	In order to provide a variance measure for mini-batch descent,  consider the function $\Lambda:\mathcal H\to \mathbb R$ given by 
	\begin{align}\label{Lambda}
		\Lambda(u):= {\sum_{j=1}^m \pi_j\big\|\xi_{j}(u) -\partial \Phi(u)^\circ\big\|^2_{\mathcal H}}.
	\end{align}
	This function is the usual quantifier of variance used to provide bounds and estimates in stochastic gradient descent algorithms. 
	We observe that $	\Var\left[ \xi_{j_k}\right]=\Lambda$ for all $k\in\mathbb N$. 
	Moreover, this function can be used to bound the gap between the gradient and mini-batch descent flows.

	\begin{prpstn}
		Let $u:[0,+\infty)\to\mathcal H$ be the solution of the gradient flow equation (\ref{gf_eq}) and $v_{\varepsilon}:[0,+\infty)\to\mathcal H$ the solution of the mini-batch descent equation (\ref{mb_gf}). Then, 
		\begin{align*}
			\big\| v_{\varepsilon}(t)-u(t)\big\|_{\mathcal H} \le \max_{j\in\{1,\dots,m\}}\pi_{j}^{-\frac{1}{2}} t^{\frac{1}{2}} \Big(\int_{0}^t \Lambda \big(u(s)\big)\, ds\Big)^{\frac{1}{2}}\quad \forall t\in[0,+\infty).
		\end{align*}
		In particular, if  $\Lambda:[0,+\infty)\to\mathbb R$ is locally integrable, then the  $v_{\varepsilon}$ is locally bounded in $L^2\big([0,+\infty);\mathcal H\big)$ for every $\varepsilon>0$.
	\end{prpstn}
	\begin{proof}
		It follows from Proposition \ref{essprop} and Corollary \ref{corobound}.
	\end{proof}
	A natural question that arises is whether the function $\Lambda\circ u: [0,+\infty)\to \mathbb R$ is locally bounded or at least locally integrable. We discuss these issues in the Appendix \ref{app:variance} as well as the dependence of $\Lambda$ on the chosen decomposition of the minimal norm subgradient.
 

Under additional assumptions on the mini-batch potentials $\Phi_{\mathcal B_{1}},\dots,\Phi_{\mathcal B_{m}}$, we can provide a better estimate in expected value. The following two theorems guarantee, with different assumptions, the convergence in expectation of the mini-batch descent flow to the gradient flow as $\varepsilon \to 0^+$.

	\begin{thrm}\label{result0}
		Let $u:[0,+\infty)\to\mathcal H$ be the solution of the gradient flow equation (\ref{gf_eq}) and $v_{\varepsilon}:[0,+\infty)\to\mathcal H$ the solution of the mini-batch descent flow equation (\ref{mb_gf}).  Suppose that $\Lambda\circ u:[0,+\infty)\to \mathcal H$ is locally integrable, and that for each $j\in\{1,\dots,m\}$,  $\Phi_{\mathcal B_{j}}$ is locally bounded in its effective domain.  Then, for each any $t\in[0,+\infty)$, $\{\varepsilon^{\frac{1}{2}}\dot v_{\varepsilon}\}_{\varepsilon>0}$ is bounded in $L^2\big([0,t);\mathcal H\big)$ and \begin{align*}
			\mathbb E\big\|v_\varepsilon(t)-u(t)\big\|^2_{\mathcal H}\le  2\varepsilon^{\frac{1}{2}} \big\|\varepsilon^{\frac{1}{2}}\dot v_\varepsilon\big\|_{L^2([0,t);\mathcal H)} \Big(\int_{0}^t \Lambda\big(u(s)\big)\, ds\Big)^{\frac{1}{2}}\quad \forall \varepsilon>0.
		\end{align*} 
		In particular, for any $T>0$ there exists $c_T>0$ such that 
		\begin{align*}
			\sup_{t\in[0,T]}\mathbb E\big\|v_\varepsilon(t)-u(t)\big\|^2_{\mathcal H}\le c_T	\varepsilon^{\frac{1}{2}}\quad \forall\varepsilon>0.
		\end{align*}
	\end{thrm}
	\begin{proof}
		It follows from Lemma \ref{lemmaintervals0} and Corollary \ref{corocc}.
	\end{proof}
	It is possible to improve the rate of convergence by replacing the boundedness assumption on the mini-batch potentials with local Lipschitz continuity. We mention that this is not a big  assumption in many cases, as for example, convex functions defined on a finite dimensional  space are automatically locally Lipschitz continuous; however this is not the case for infinite dimensional spaces (due to the celebrated Hahn-Banach Theorem there exist unbounded linear functionals). Also note that there are convex functions that are locally bounded in its effective domain, but not locally Lipschitz, even if the underlying Hilbert space is finite  dimensional\footnote{The function $\psi:\mathbb R\to\mathbb R\cup\{+\infty\}$, given by $\psi(u):=-2\sqrt{u}$ if $u\ge0$ and $+\infty$ otherwise, furnishes an example.}.
 
	\begin{thrm}\label{resultcontinuouscase}
			Let $u:[0,+\infty)\to\mathcal H$ be the solution of the gradient flow equation (\ref{gf_eq}) and $v_{\varepsilon}:[0,+\infty)\to\mathcal H$ the solution of the mini-batch descent flow equation (\ref{mb_gf}).  Suppose that $\Lambda\circ u:[0,+\infty)\to \mathcal H$ is locally integrable, and that for each $j\in\{1,\dots,m\}$,  $\Phi_{\mathcal B_{j}}$ is locally Lipschitz in its effective domain.  Then, for any $t\in[0,+\infty)$, $\{\dot v_{\varepsilon}\}_{\varepsilon>0}$ is bounded in $L^2\big([0,t);\mathcal H\big)$ and \begin{align*}
			\mathbb E\big\|v_\varepsilon(t)-u(t)\big\|^2_{\mathcal H}\le  2\varepsilon \big\|\dot v_\varepsilon\big\|_{L^2([0,t);\mathcal H)} \Big(\int_{0}^t \Lambda\big(u(s)\big)\, ds\Big)^{\frac{1}{2}}\quad \forall \varepsilon>0.
		\end{align*} 
		In particular, for any $T>0$ there exists $c_T>0$ such that 
		\begin{align*}
			\sup_{t\in[0,T]}\mathbb E\big\|v_\varepsilon(t)-u(t)\big\|^2_{\mathcal H}\le c_T	\varepsilon\quad \forall\varepsilon>0.
		\end{align*}
	\end{thrm}
	\begin{proof}
		It follows from Lemma \ref{lemmaintervals0} and Theorem \ref{thrcc}.
	\end{proof}
Previous results will be employed in more particular settings, especially projected dynamics resulting from constrained optimization (Section \ref{projdyn}), and sparse inversion problems arising in data science (Section \ref{Sparseinv}).

	\subsection{Mini-batch minimizing movement}\label{sec:RMM}
For some applications, such as domain decomposition \cite{Eisenmann2024randomized}, it is preferable to remove assumptions on the potential, such as local boundedness. The possible lack of absolute continuity of the solution of the mini-batch descent flow generates several difficulties in the analysis; therefore, it is more convenient to analyze it in a discrete setting. We begin by introducing a randomized variant of the so-called minimizing movement scheme; this is essentially a stochastic proximal point algorithm.

\smallbreak 	
	
	It can be proved that there exists a unique $w_1\in\argmin_{w\in\mathcal H}\{\Phi_{\mathcal B_{j_1}}(w)+\frac{1}{2\varepsilon}\big\|w-u_{0}\big\|_{\mathcal H}^2\}$; see, e.g., \cite[Proposition 17.2.1]{SoBV}. This procedure can continue in an iterative way, yielding a sequence  $\{w_k\}_{k\in\mathbb N}$ satisfying
	\begin{align}\label{gf_mm0}
		w_{k+1}\in\argmin_{w\in\mathcal H}\left\lbrace \Phi_{\mathcal B_{j_{k+1}}}(w)  + \displaystyle\frac{1}{2\varepsilon}\big\|w-w_{k}\big\|_{\mathcal H}^2 \right\rbrace\quad \forall k\in\mathbb N.
	\end{align}
	Consider the function $w_{\varepsilon}:[0,+\infty)\to \mathcal H$ given by $w_{\varepsilon}(t):=w_{k_{t}}$. We see that for each $k\in\mathbb N$,
	\begin{align}\label{wepsilon}
		w_{\varepsilon}(t)= w_{k}\quad \forall t\in[t_{k-1},t_k). 
	\end{align}
	We say that $w : [0,+\infty) \to \mathcal H$ is a mini-batch minimizing movement associated with the mini-batch potentials $\Phi_{\mathcal B_1},\dots,\Phi_{\mathcal B_{m}}$ if for all $t\in[0,+\infty)$,
	\[ \mathbb E\big\|w_\varepsilon(t) - w(t)\big\|_{\mathcal H}^2\, {\longrightarrow}\, 0 \quad \text{as}\quad \varepsilon\to0^+. \]

Below, \Cref{alg:minimizing_movement} illustrate the construction of the sequence $\{w_k\}_{k=1}^K$, for some fixed $K>0$. Here, for the sake of clarity, we illustrate the construction for a fixed $K$. However, the same iteration can be done for $K=\infty$.

\begin{algorithm}[H]
\caption{Mini-Batch Minimizing Movement (finite sequence)}\label{alg:minimizing_movement}
\begin{algorithmic}[1]
\Require Initial data $u_0$, fixed $\varepsilon$, iterations $K$, batches $\{\mathcal{B}_1, \dots, \mathcal{B}_m\}$, probabilities $\{\pi_1, \dots, \pi_m\}$

\State Initialize $w_0 \gets u_0$

\For{$k \gets 1$ to $K$}
    \State Sample a batch index $j_k$ according to probabilities $\{\pi_1, \dots, \pi_m\}$
    \State Let $\mathcal{B} \gets \mathcal{B}_{j_k}$
    \State Let $\Phi_{\mathcal{B}}(u) \gets \frac{1}{|B|}\sum_{i\in B} \Phi_{i}(u)$
    \State Solve $w_{k}\in\argmin_{w\in\mathcal H}\left\lbrace \Phi_{\mathcal B}(w)  + \displaystyle\frac{1}{2\varepsilon}\big\|w-w_{k-1}\big\|_{\mathcal H}^2 \right\rbrace$
    \State Let $w_{\varepsilon}(t)\gets w_{k}$ in $[t_{k-1},t_k)$
\EndFor

\Return $w_{\varepsilon}$
\end{algorithmic}
\end{algorithm}

	The following result gives sufficient conditions for a gradient flow solution to be a mini-batch minimizing movement and, moreover, provides an error estimate.
	\begin{thrm}\label{th:conv_mm}
		Let $u:[0,+\infty)\to\mathcal H$ be the solution of the gradient flow equation (\ref{gf_eq}) and $w_{\varepsilon}:[0,+\infty)\to\mathcal H$ the function generated by the proximal sequence (\ref{gf_mm0}). Assume that $\Lambda\circ u:[0,+\infty)\to \mathcal H$ is locally  bounded. If $u$ belongs to $C^1\big([0,+\infty);\mathcal H\big)$, then $u$ is a mini-batch minimizing movement.
		Moreover, if there exists $\alpha\in(0,1]$ such that $\dot u$ is locally $\alpha$-H\"older continuous, then for any $T>0$ there exists $c_T>0$ such that 
		\begin{align*}
			\sup_{t\in[0,T]}\mathbb E\big\|w_\varepsilon(t) - u(t)\big\|_{\mathcal H}^2\le c_T\big(\varepsilon^{2\alpha}+\varepsilon^2+\varepsilon\big)\quad \forall\varepsilon>0.
		\end{align*}
	\end{thrm}
	\begin{proof}
		It follows from Corollary \ref{corfin}.
	\end{proof}
	
   We highlight that to ensure the convergence of the random minimizing movement to the gradient flow, the above result requires only regularity assumptions on the gradient flow and not on the random scheme.

	\subsection{Asymptotic behavior}
	We have seen that under some hypotheses, the trajectory of mini-batch descent closely approximates the path of the gradient flow. It is natural to expect that one can say something about the minimization of the full potential. The following result analyzes the asymptotic behavior of the mini-batch descent flow.
	\begin{thrm}
		Let $u:[0,+\infty)\to \mathcal H$ be the solution of the gradient flow equation (\ref{gf_eq}) and $v_{\varepsilon}:[0,+\infty)\to \mathcal H$ the solution of the mini-batch descent flow equation (\ref{mb_gf}).  Suppose that $\Lambda\circ u:[0,+\infty)\to \mathcal H$ is locally integrable and that there exists $\alpha\in(0,1]$ such that, for each $j\in\{1,\dots,m\}$, $\Phi_{\mathcal B_j}$ is locally $\alpha$-H\"older continuous in its effective domain. Then, the following statements hold. 
		\begin{itemize}
			\item[(i)]  If  $\dom\partial\Phi_{\mathcal B_j}\subset \dom\Phi$ for all $j\in\{1,\dots,m\}$, then for every $\eta>0$ there exists $T>0$ such that 
			\begin{align*}
				\mathbb E\Phi(v_{\varepsilon}(T)) \le \inf_{v\in\mathcal H}\Phi(v) + \eta \quad \text{for all $\varepsilon>0$ small enough.}
			\end{align*}
			
			\item[(ii)] If $\Phi:\mathcal H\to\mathbb R\cup\{+\infty\}$ is inf-compact, then there exists $u^*\in\argmin_{v\in\mathcal H}\Phi(v)$ satisfying that for every $\eta>0$ there exists $T>0$ such that
			\begin{align*}
				\mathbb E\big\|v_\varepsilon(T) - u^*\big\|_{\mathcal H}^2\le \eta \quad \text{for all $\varepsilon>0$ small enough.}
			\end{align*}
		\end{itemize}
	\end{thrm}
	\begin{proof}
		This is a particular case of Theorem \ref{Asyncc}.
	\end{proof}

	For the minimizing movement scheme,  it is also possible to provide an analogous result on the asymptotic behavior of trajectories. 
	\begin{thrm}\label{th:asym_beh}
		Let $u:[0,+\infty)\to \mathcal H$ be the solution of the gradient flow equation (\ref{gf_eq}) $w_{\varepsilon}:[0,+\infty)\to\mathcal H$ the function generated by the proximal sequence (\ref{gf_mm0}).  Suppose that  $\Phi:\mathcal H\to\mathbb R\cup\{+\infty\}$ is inf-compact, $\Lambda\circ u:[0,+\infty)\to \mathcal H$ is locally bounded and that $\dot u$ is continuous. Then, there exists $u^*\in\argmin_{v\in\mathcal H}\Phi(v)$ satisfying that for every $\eta>0$ there exists $T>0$ such that
		\begin{align*}
			\mathbb E\big\|w_\varepsilon(T) - u^*\big\|_{\mathcal H}^2\le \eta \quad \text{for all $\varepsilon>0$ small enough.}
		\end{align*}
	\end{thrm}
	\begin{proof}
		The proof is analogous to that of Theorem \ref{Asyncc}.
	\end{proof}

 \section{Well-posedness  and convergence of mini-batch descent}\label{sec:well_posedness_convergence}

In the remainder of this paper, we assume that $u_0 \in \dom(\partial\Phi)$, condition \eqref{assu:contention_domains} holds, and the sum rule \eqref{assu:sum_rule} is satisfied.

	\subsection{Existence and uniqueness}
	The existence of the mini-batch descent flow (\ref{mb_gf}) is established through an induction argument that ensures existence within each interval. The following argument encapsulates this process.
	\begin{lmm}\label{lemwpsgf}
		Let $k\in\mathbb N$. There exists a unique  continuous function $v_{k}:[0,t_{k}]\to\mathcal H$ such that
		\begin{itemize}
			\item[(i)] $v_{k}(0) = u_{0}$;
			\item[(ii)] $v_{k}$ is locally absolutely continuous in $(t_{l-1},t_l)$ for all $l\in \{1,\dots, k\}$;
			\item[(iii)] $0\in \dot{v}_{k} + \partial\Phi_{\mathcal B_{j_l}}(v_{k})$ a.e. in $[t_{l-1},t_{l}]$ for all $l\in\{1,\dots,k\}$.
		\end{itemize}
	\end{lmm}
	\begin{proof}
		Let $S$ denote the set of all natural numbers $k\in\mathbb N$ such that there exists a continuous function $v_{k}:[0,t_{k}]\to\mathcal H$ satisfying items $(i)$-$(iii)$. We argue by mathematical induction. Since $u_0\in\dom \Phi\subset\dom \Phi_{\mathcal B_{j_1}}$, by \cite[Theorem 17.2.3]{SoBV} there exists a unique  continuous function $w_{1}:[0,t_1]\to \mathcal H$ such that $w_1(0)=u_0$, $w_{1}$ is locally absolutely continuous in $(0,t_1)$, and $0\in w_1(t)+\partial \Phi_{\mathcal B_{j_1}}(w_1(t))$ for a.e. $t\in[0,t_1]$; thus $1\in S$.

		We now proceed with the inductive step. Let $k\in S$, then there exists a continuous function  $v_{k}:[0,t_{k}]\to\mathcal H$ satisfying  items $(i)$-$(iii)$.
		Since $v_{k}(t_{k})\in \dom\Phi_{\mathcal B_{j_k}}\subset \overline{\dom\Phi}\subset\overline{\dom\Phi_{\mathcal B_{j_{k+1}}}}$, by \cite[Theorem 17.2.3]{SoBV}, there exists a unique  continuous function $w_{k+1}:[t_{k},t_{k+1}]\to\mathcal H$ such that $w_{k+1}(t_{k})=v_{k}(t_{k})$, $w_{k+1}$ is locally absolutely continuous in $(t_{k},t_{k+1})$ and $0\in\dot w_{k+1}(t)+\partial\Phi_{\mathcal B_{j_{k+1}}}(w_{k+1}(t))$ for a.e. $t\in[t_{k},t_{k+1}]$. Define $v_{k+1}:[0,t_{k+1}]\to \mathcal H$ by
		\begin{align*}
			v_{k+1}(t):=
			\left\{ \begin{array}{lcc} 
				v_{k}(t) &\text{if}& t\in[0,t_{k}],\\ w_{k+1}(t) &\text{if}& t\in(t_{k},t_{k+1}].\end{array} \right.
		\end{align*}
		By construction, $v_{k+1}$ is a continuous function satisfying items $(i)$-$(iii)$; thus $k+1\in S$.  This completes the induction, and hence $S=\mathbb N$.
	\end{proof}
	
	\begin{thrm}\label{Existence}
		There exists a unique continuous function $v_{\varepsilon}:[0, +\infty)\to\mathcal H$ such that  
		\begin{itemize}
			\item[(i)] $v_{\varepsilon}(0) = u_{0}$;
			\item[(ii)] $v_{\varepsilon}$ is locally absolutely continuous in $(t_{k-1},t_k)$ for all $k\in\mathbb N$;
			\item[(iii)] $0\in \dot{v}_{\varepsilon} + \partial\Phi_{\mathcal B_{j_k}}(v_{\varepsilon})$ a.e. in $[t_{k-1},t_{k}]$ for all $k\in\mathbb N$.
		\end{itemize}
	\end{thrm}
	\begin{proof}
		By Lemma \ref{lemwpsgf}, for each $k\in\mathbb N$ there exists a unique continuous function $v_{k}:[0,t_k]\to\mathcal H$ such that    $v_{k}(0)=u_0$, $v_{k}$ is locally absolutely continuous in $(t_{k-1},t_{k})$ and $0\in\dot v_{k}(t)+\partial\Phi_{j_k}(v_{k}(t))$ for a.e. $t\in[t_{k-1},t_{k}]$. Define $v_{\varepsilon}:[0,+\infty)\to\mathcal H$ by $v_\varepsilon(t):=v_{k_t}(t)$; clearly $v_\varepsilon$ is continuous, and satisfies items $(i)$-$(iii)$. Its uniqueness follows from the uniqueness of the functions $\{v_{k}\}_{k\in\mathbb N}$.
	\end{proof}
	
	\begin{prpstn}\label{Regularity}
		Let $v_{\varepsilon}:[0,+\infty)\to\mathcal H$ be the solution of the mini-batch descent flow equation (\ref{mb_gf}).
		Then, 
		\begin{itemize}
			\item[(i)]  if $\dom\partial\Phi_{\mathcal B_j}\subset \dom\Phi$ for all $j\in\{1,\dots,m\}$, then $v_{\varepsilon}:[0,+\infty)\to\mathcal H$ is locally absolutely continuous;
			
			\item[(ii)]  if $\dom \partial\Phi_{\mathcal B_j}\subset \dom\partial\Phi$ for all $j\in\{1,\dots,m\}$, then $v_{\varepsilon}:[0,+\infty)\to\mathcal H$ is locally Lipschitz continuous.
		\end{itemize}
	\end{prpstn}	
	\begin{proof}
		Given $k\in\mathbb N$, by \cite[Theorem 4.11]{Barbu2010}, if $v_{\varepsilon}(t_{k-1})\in\dom\Phi\subset\dom  \Phi_{\mathcal B_{j_k}}$, then the restriction of $v_\varepsilon$ to the subinterval $[t_{k-1},t_k]$ is absolutely continuous; we conclude that $v_{\varepsilon}$ is absolutely continuous in any compact subset of $[0,+\infty)$. This proves item $(i)$; the proof of item $(ii)$  follows the same argument replacing \cite[Theorem 4.11]{Barbu2010} by \cite[Theorem 17.2.2]{SoBV}.
	\end{proof}

	\subsection{Convergence}
	\begin{prpstn}\label{essprop}
		Let $u:[0,+\infty)\to\mathcal H$ be the solution of the gradient flow equation (\ref{gf_eq}) and $v_{\varepsilon}:[0,+\infty)\to\mathcal H$ the solution of the mini-batch descent flow equation (\ref{mb_gf}). Then, 
		\begin{align*}
			\big\| v_{\varepsilon}(t)-u(t)\big\|_{\mathcal H} \le \max_{j\in\{1,\dots,m\}}\pi_{j}^{-\frac{1}{2}} t^{\frac{1}{2}} \Big(\int_{0}^t \Lambda \big(u(s)\big)\, ds\Big)^{\frac{1}{2}}\quad \forall t\in[0,+\infty).
		\end{align*}
	\end{prpstn}
	\begin{proof}
		Let $k\in\mathbb N$. By monotonicity of the subdifferential, 
		\begin{align*}
			\left\langle -\dot v_{\varepsilon}(\tau) - \xi_{j_{k}}(u(\tau)),\, v_\varepsilon(\tau) - u(\tau) \right\rangle \ge 0\quad \text{for a.e.  $\tau\in(t_{k-1},t_k)$},
		\end{align*}
		where the functions $\xi_{j_{k}}$ are given as in \eqref{eq:represent_subdif}. Using that $\dot u = -\partial\Phi(u)^{\circ}$ a.e. in $[0,+\infty)$, this can be rewritten as 
		\begin{align*}
			0\le 	\left\langle -\dot v_{\varepsilon}(\tau) + \dot u(\tau),\, v_{\varepsilon}(\tau) - u(\tau) \right\rangle + \left\langle \partial \Phi(u(\tau))^\circ - \xi_{j_{k}}(u(\tau)),\, v_{\varepsilon}(\tau) - u(\tau) \right\rangle,
		\end{align*}
		for a.e. $\tau\in(t_{k-1}, t_k)$. This implies
		\begin{align*}
			\frac{1}{2}\frac{d}{d\tau}\big\|v_\varepsilon(\tau)-u(\tau)\big\|^2_{\mathcal H}&\le  \big\|  \partial \Phi(u(\tau))^\circ - \xi_{j_{k}}(u(\tau))\big\|_{\mathcal H} \big\|  v_\varepsilon(\tau) - u(\tau)\big\|_{\mathcal H}\\
			&\le \max_{j\in\{1,\dots,m\}}\pi_{j}^{-\frac{1}{2}}\sqrt{\Lambda\big(u(\tau)\big)} \big\|  v_\varepsilon(\tau) - u(\tau)\big\|_{\mathcal H},
		\end{align*}
		for a.e. $\tau\in(t_{k-1},t_{k})$. Due to Theorem \ref{Existence}, the function  $t\longmapsto \big\|v_{\varepsilon}(t)-u(t)\big\|_{\mathcal H}$ is locally absolutely continuous in $(t_{k-1},t_k)$. Then, we can employ Gr\"onwall's inequality  \cite[Theorem 5.2.2]{DistributionTheory} to conclude that
		\begin{align*}
			\big\|v_\varepsilon(t)-u(t)\big\|_{\mathcal H}\le \big\|v_\varepsilon(s)-u(s)\big\|_{\mathcal H} + \max_{j\in\{1,\dots,m\}}\pi_{j}^{-\frac{1}{2}}\int_{s}^{t} \sqrt{\Lambda\big(u(\tau)\big)} \, d\tau \quad\forall s,t\in(t_{k-1}, t_{k}).
		\end{align*}
		Set $\gamma:=\max_{j\in\{1,\dots,m\}}\pi_{j}^{-\frac{1}{2}}$. By continuity of $v_\varepsilon$ and $u$, we conclude  that
		\begin{align}\label{genin0}
			\big\|v_\varepsilon(t)-u(t)\big\|_{\mathcal H}\le \big\|v_\varepsilon(t_{k-1})-u(t_{k-1})\big\|_{\mathcal H} + \gamma\int_{t_{k-1}}^{t_{}} \sqrt{\Lambda\big(u(\tau)\big)} \, d\tau\quad \forall t\in[t_{k-1},t_k].
		\end{align}
		Since (\ref{genin0}) holds for arbitrary $k\in\mathbb N$, this implies
		\begin{align*}
			\big\|v_\varepsilon(t)-u(t)\big\|_{\mathcal H}\le\gamma \Big[ \int_{t_{k_t-1}}^{t} \sqrt{\Lambda\big(u(\tau)\big)} \, d\tau + \sum_{l=1}^{k_t-1}\int_{t_{l-1}}^{t_l}\sqrt{\Lambda\big(u(\tau)\big)}\, ds\Big]=\gamma\int_{0}^{t} \sqrt{\Lambda\big(u(\tau)\big)} \, d\tau, 
		\end{align*}
		for all $t\in[0,+\infty)$. The result follows then from H\"older's inequality.
	\end{proof}
	
	\begin{crllr}\label{corobound}
		Let $u:[0,+\infty)\to\mathcal H$ be the solution of the gradient flow equation (\ref{gf_eq}) and $v_{\varepsilon}:[0,+\infty)\to\mathcal H$ the solution of the mini-batch descent flow  equation (\ref{mb_gf}). Then, 
		\begin{align*}
			\big\| v_{\varepsilon}(t)\big\|_{\mathcal H}\le  \max_{j\in\{1,\dots,m\}}\pi_{j}^{-\frac{1}{2}} t^{\frac{1}{2}} \Big(\int_{0}^t \Lambda \big(u(s)\big)\, ds\Big)^{\frac{1}{2}} +	t\big\| \partial \Phi(u_0)^\circ\big\|_{\mathcal H} + \big\|u_0\big\|_{\mathcal H} \quad\forall t\in[0,+\infty).
		\end{align*}
	\end{crllr}
	\begin{proof}
		Since $u$ is Lipschitz continuous with constant $\big\|\partial \Phi(u_0)^\circ\big\|_{\mathcal H}$ in $[0,+\infty)$, we have that
		\begin{align*}
			\big\| v_{\varepsilon}(t)\big\|_{\mathcal H}\le 	\big\| v_{\varepsilon}(t)-u(t)\big\|_{\mathcal H} + 	\big\| u(t)-u_0\big\|_{\mathcal H} + \big\|u_0\big\|_{\mathcal H}\le 	\big\| v_{\varepsilon}(t)-u(t)\big\|_{\mathcal H} + 	t\big\| \partial \Phi(u_0)^\circ\big\|_{\mathcal H} + \big\|u_0\big\|_{\mathcal H},
		\end{align*}
		for all $t\in[0,+\infty)$. The result follows then from Proposition \ref{essprop}.
	\end{proof}
	
	\begin{lmm}\label{lemmaintervals0}
		Let $u:[0,+\infty)\to\mathcal H$ be the solution of the gradient flow equation (\ref{gf_eq}) and $v_\varepsilon:[0,+\infty)\to\mathcal H$ the solution of the mini-batch descent flow equation (\ref{mb_gf}). Suppose that $\Lambda\circ u:[0,+\infty)\to \mathcal H$ is locally integrable. Then, the following statements hold. 
	\begin{itemize}
			\item[(i)] If, for each $j\in\{1,\dots,m\}$,  $\Phi_{\mathcal B_{j}}$ is locally bounded in its effective domain, then for any $T>0$,  $\{\varepsilon^{\frac{1}{2}}\dot v_{\varepsilon}\}_{\varepsilon>0}$ is bounded in $L^2\big([0,T];\mathcal H\big)$.
			
			\item[(ii)] If there exists $\alpha\in(0,1]$ such that, for each $j\in\{1,\dots,m\}$, $\Phi_{\mathcal B_j}$ is locally $\alpha$-H\"older continuous in its effective domain, then for any $T>0$, $\{\varepsilon^{\frac{1-\alpha}{2-\alpha}}\dot v_{\varepsilon}\}_{\varepsilon>0}$ is bounded in $L^2\big([0,T];\mathcal H\big)$.
		\end{itemize}
		
	\end{lmm}
	\begin{proof}
		Let $T>0$ be given. Both items $(i)$ and $(ii)$ can be treated at the same time, allowing $\alpha$ to be zero; in this way a function is locally bounded if and only if it is locally $0$-H\"older continuous. Define 
		\begin{align*}
			r_{T}:= \max_{j\in\{1,\dots,m\}}\pi_{j}^{-\frac{1}{2}} T^{\frac{1}{2}} \Big(\int_{0}^T \Lambda \big(u(s)\big)\, ds\Big)^{\frac{1}{2}} +T\big\| \partial \Phi(u_0)^\circ\big\|_{\mathcal H} + \big\|u_0\big\|_{\mathcal H}.
		\end{align*}
		Since $\Lambda\circ u\in L^1\big([0,T];\mathbb R\big)$, $r_T<+\infty$. By assumption, there exists $L_{T}>0$ such that, for any $j\in\{1,\dots,m\}$, 
		\begin{align*}
			\frac{|\Phi_{\mathcal B_{j_{}}}(v)-\Phi_{\mathcal B_{j_{}}}(w)|}{\big\| v - w\big\|_{\mathcal H}^\alpha}\le L_T\quad \text{for all $v,w\in\mathbb B_{\mathcal H}\big(0,r_{T}\big)$ satisfying $v\neq w$.}
		\end{align*}
		Let $k\in\mathbb N$.
		By \cite[Theorem 17.2.3]{SoBV}, $\Phi_{\mathcal B_{j_{k}}}\circ v_{\varepsilon}$ is locally  absolutely continuous in $(t_{k-1},t_k)$, and
		\begin{align}\label{geneq12}
			\frac{d}{d\tau}\Phi_{\mathcal B_{j_{k}}}\big(v_{\varepsilon}(\tau)\big) = - \big\| \dot v_{\varepsilon}(\tau) \big\|_{\mathcal H}^2\quad \text{for a.e. $\tau\in(t_{k-1},t_{k})$.}
		\end{align}
		From this and Corollary \ref{corobound}, for all $s,t\in (t_{k-1},t_k)$,
		\begin{align*}
			\int_{s}^t\big\| \dot v_{\varepsilon}(\tau) \big\|_{\mathcal H}^2\,d\tau &= \Phi_{\mathcal B_{j_k}}(v_\varepsilon(s)) - 	\Phi_{\mathcal B_{j_k}}(v_\varepsilon(t))\le L_{T} \big\| v_{\varepsilon}(s) - v_{\varepsilon}(t)\big\|_{\mathcal H}^\alpha\\
			&\le L_T\Big(\int_{s}^t\big\| \dot v_{\varepsilon}(\tau)\big\|_{\mathcal H}\,d\tau\Big)^\alpha\le  L_T (t-s)^{\frac{\alpha}{2}}\Big(\int_{s}^t\big\| \dot v_{\varepsilon}(\tau)\big\|_{\mathcal H}^2\,d\tau\Big)^\frac{\alpha}{2}.
		\end{align*}
		From where we conclude that, for all $s,t\in (t_{k-1},t_k)$,
		\begin{align}\label{bound0}
			\int_{s}^t\big\| \dot v_{\varepsilon}(\tau) \big\|_{\mathcal H}^2\,d\tau\le L_{T}^{\frac{2}{2-\alpha}} (t-s)^{\frac{\alpha}{2-\alpha}}=L_{T}^{\frac{2}{2-\alpha}}{(t-s)^{\frac{2(\alpha-1)}{2-\alpha}}}(t-s).
		\end{align}
		Since $k\in\mathbb{N}$ was arbitrary, we conclude from (\ref{bound0}) that
		\begin{align*}
			\int_{0}^T\big\| \dot v_{\varepsilon}(\tau) \big\|_{\mathcal H}^2\,d\tau&\le \int_{t_{k_{T}-1}}^T\big\| \dot v_{\varepsilon}(\tau) \big\|_{\mathcal H}^2\,d\tau + \sum_{l=1}^{k_T-1}\int_{t_{l-1}}^{t_l}\big\| \dot v_{\varepsilon}(\tau) \big\|_{\mathcal H}^2\,d\tau\le L_{T}^{\frac{2}{2-\alpha}} \varepsilon^{\frac{2(\alpha-1)}{2-\alpha}} \Big[T - t_{k_T-1} + \sum_{l=1}^{k_T-1}(t_{l}-t_{l-1})\Big].
		\end{align*}
		We can then conclude that $\displaystyle\int_{0}^T\big\|\varepsilon^{\frac{1-\alpha}{2-\alpha}} \dot v_{\varepsilon}(\tau) \big\|_{\mathcal H}^2\,d\tau = \varepsilon^{\frac{2(1-\alpha)}{2-\alpha}}	\int_{0}^T\big\| \dot v_{\varepsilon}(\tau) \big\|_{\mathcal H}^2\,d\tau \le  L_{T}^{\frac{2}{2-\alpha}} T.$
	\end{proof}
	
	\begin{thrm}\label{thrcc}
		Let $u:[0,+\infty)\to\mathcal H$ be the solution of the gradient flow equation (\ref{gf_eq}) and $v_{\varepsilon}:[0,+\infty)\to\mathcal H$ the solution of the mini-batch descent flow equation (\ref{mb_gf}). Then,
		\begin{align*}
			\mathbb E\big\|v_\varepsilon(t)-u(t)\big\|^2_{\mathcal H}\le  2\varepsilon \Big(\int_{0}^t \mathbb E\big\|\dot v_{\varepsilon}(\tau)\big\|_{\mathcal H}^2\,d\tau\Big)^{\frac{1}{2}} \Big(\int_{0}^t \Lambda\big(u(s)\big)\, ds\Big)^{\frac{1}{2}}\quad \forall t\in[0,+\infty).
		\end{align*}
	\end{thrm}
	\begin{proof}
		Let $k\in\mathbb N$. By monotonicity of the subdifferential, 
		\begin{align*}
			\left\langle -\dot v_{\varepsilon}(\tau) - \xi_{j_{k}}(u(\tau)),\, v_\varepsilon(\tau) - u(\tau) \right\rangle \ge 0\quad \text{for a.e.  $\tau\in(t_{k-1},t_k)$}.
		\end{align*}
		Using that $\dot u = -\partial\Phi(u)^{\circ}$ a.e. in $[0,+\infty)$, this can be rewritten as 
		\begin{align*}
			0\le 	\left\langle -\dot v_{\varepsilon}(\tau) + \dot u(\tau),\, v_{\varepsilon}(\tau) - u(\tau) \right\rangle + \left\langle \partial \Phi(u(\tau))^\circ - \xi_{j_{k}}(u(\tau)),\, v_{\varepsilon}(\tau) - u(\tau) \right\rangle,
		\end{align*}
		for a.e. $\tau\in(t_{k-1}, t_k)$.
		This implies that for a.e. $\tau\in(t_{k-1},t_{k})$,
		\begin{align}\label{gen1}
			\frac{1}{2}\frac{d}{d\tau}\big\|v_\varepsilon(\tau)-u(\tau)\big\|^2_{\mathcal H}\le  \left\langle   \partial \Phi(u(\tau))^\circ - \xi_{j_{k}}(u(\tau)),\,  v_\varepsilon(\tau) - u(\tau)\right\rangle.
		\end{align}
		Let $\chi:[0,+\infty)\to \mathcal H$ be given by  $\chi(t)=\partial \Phi(u(t))^\circ - \xi_{j_{k}}\left(u(t)\right)$. Since, $\mathbb E\chi(\tau) = 0$ for all $\tau\in(t_{k-1},t_k)$,
		\begin{align*}
			\mathbb E\left[\left\langle \chi(\tau),\, v_\varepsilon(t_{k_{}-1}) - u(\tau) \right\rangle\right]=\left\langle \mathbb E\chi(\tau), \mathbb E v_{\varepsilon}(t_{k_{}-1}) - u(\tau) \right\rangle =  0\quad \forall\tau\in(t_{k-1},t_k).
		\end{align*}
		Combining this with (\ref{gen1}), we get, for a.e. $\tau\in(t_{k-1},t_k)$,
		\begin{align}\label{derc1}
			\frac{1}{2}\frac{d}{d\tau}\mathbb E\big\|v_\varepsilon(\tau)-u(\tau)\big\|^2_{\mathcal H}\le \mathbb E\left[\left\langle \chi(\tau),\, v_{\varepsilon}(\tau) - v_{\varepsilon}(t_{k-1}) \right\rangle \right]\le  \sqrt{\mathbb E\big\|\chi(\tau)\big\|^2_{\mathcal H}}  \sqrt{\mathbb E\big\|v_\varepsilon(\tau)-v_\varepsilon(t_{k-1})\big\|^2_{\mathcal H}}.
		\end{align}
		Observe that $\mathbb E\big\|\chi(t)\big\|_{\mathcal H}^2 = \Lambda\big(u(t)\big)$ for all $t\in[0,+\infty)$. Now, since $v_{\varepsilon}$ is locally absolutely continuous in $(t_{k-1},t_k)$, so is $t\longmapsto$ $\big\|v_{\varepsilon}(t)-u(t)\big\|_{\mathcal H}$; thus, integrating yields
		\begin{align*}
			\mathbb E\big\|v_\varepsilon(t)-u(t)\big\|^2_{\mathcal H} \le \mathbb E\big\|v_\varepsilon(s)-u(s)\big\|^2_{\mathcal H} + 2\int_{s}^t \sqrt{\Lambda\big(u(\tau)\big)}\sqrt{\mathbb E\big\|v_\varepsilon(\tau)-v_\varepsilon(t_{k-1})\big\|^2_{\mathcal H}}\, d\tau\quad \forall s,t\in(t_{k-1},t_{k}).
		\end{align*}
		Hence, by H\"older's inequality, for all $t\in[t_{k-1},t_{k}]$,
		\begin{align}\label{ineq}
			\mathbb E\big\|v_\varepsilon(t)-u(t)\big\|^2_{\mathcal H} \le \mathbb E\big\|v_\varepsilon(t_{k-1})-u(t_{k-1})\big\|^2_{\mathcal H} + 2\Big(\int_{t_k}^t \Lambda\big(u(\tau)\big)\, d\tau\Big)^{\frac{1}{2}}\Big(\int_{t_k}^t\mathbb E\big\| v_{\varepsilon}(\tau) - v_{\varepsilon}(t_{k-1}) \big\|_{\mathcal H}^2\, d\tau\Big)^{\frac{1}{2}}.
		\end{align}
		For each $l\in\mathbb N$, let $h_l:[t_{l-1},t_l]\to\mathbb R$ be given by $h_{l}(t):=\mathbb E\big\| v_{\varepsilon}(t) - v_{\varepsilon}(t_{l-1})\big\|_{\mathcal H}^2$.
		Since $k\in\mathbb N$ was arbitrary, by  (\ref{ineq}) and H\"older's inequality, 
		\begin{align}
			\nonumber\mathbb E\big\|v_\varepsilon(t)-u(t)\big\|^2_{\mathcal H} &\le 2 \Big(\int_{t_{k_t-1}}^t \Lambda\circ u\Big)^{\frac{1}{2}}\Big(\int_{t_{k_t-1}}^t h_{k_t}\Big)^{\frac{1}{2}} + 2 \sum_{l=1}^{k_{t}-1} \Big(\int_{t_{l-1}}^{t_l} \Lambda\circ u\Big)^{\frac{1}{2}}\Big(\int_{t_{l-1}}^{t_l} h_{l}\Big)^{\frac{1}{2}}\\
			\nonumber&\le 2 \Big(\int_{t_{k_t-1}}^{t}\Lambda\circ u +  \sum_{l=1}^{k_{t}-1} \int_{t_{l-1}}^{t_l} \Lambda\circ u\Big)^{\frac{1}{2}}  \Big(\int_{t_{k_t-1}}^{t} h_{k_t} +  \sum_{l=1}^{k_{t}-1} \int_{t_{l-1}}^{t_l} h_l\Big)^{\frac{1}{2}}\\
			&= 2\Big(\int_{0}^{t}\Lambda\circ u \Big)^{\frac{1}{2}}  \Big(\int_{t_{k_t-1}}^{t} h_{k_t} +  \sum_{l=1}^{k_{t}-1} \int_{t_{l-1}}^{t_l} h_l\Big)^{\frac{1}{2}},\label{forfinaline}
		\end{align}
		for every $t\in[0,+\infty)$. Now observe that, for each $l\in\mathbb{N}$ 
		\begin{align*}
			h_{l}(t)\le \mathbb E\Big(\int_{t_{l-1}}^t\big\|\dot v_{\varepsilon}\big\|_{\mathcal H}\Big)^2\le (t-t_{l})\, \mathbb E\int_{t_{l-1}}^t\big\|\dot v_{\varepsilon}\big\|_{\mathcal H}^2\le \varepsilon\int_{l-1}^t\mathbb E\big\|\dot v_{\varepsilon}\big\|_{\mathcal H}^2\, d\tau\quad\forall t\in[t_{l-1},t_l).
		\end{align*}
		From this, we get that for every $t\in[0,+\infty)$,
		\begin{align*}
			\int_{t_{k_t-1}}^{t} h_{k_t} +  \sum_{l=1}^{k_{t}-1} \int_{t_{l-1}}^{t_l} h_l&\le	\varepsilon\Big[\int_{t_{k_t-1}}^{t} \Big(\int_{t_{k_t-1}}^s \mathbb E\big\|\dot v_{\varepsilon}\big\|_{\mathcal H}^2\Big)\, ds +  \sum_{l=1}^{k_{t}-1} \int_{t_{l-1}}^{t_l} \Big(\int_{t_{l-1}}^{s} \mathbb E\big\|\dot v_{\varepsilon}\big\|_{\mathcal H}^2\Big)\, ds\Big]\\
			&\le \varepsilon\Big[\int_{t_{k_t-1}}^{t} \Big(\int_{t_{k_t-1}}^t \mathbb E\big\|\dot v_{\varepsilon}\big\|_{\mathcal H}^2\Big)\, ds +  \sum_{l=1}^{k_{t}-1} \int_{t_{l-1}}^{t_l} \Big(\int_{t_{l-1}}^{t_l} \mathbb E\big\|\dot v_{\varepsilon}\big\|_{\mathcal H}^2\Big)\, ds\Big]\\
			&\le \varepsilon\Big[(t-t_{k_t-1})\int_{t_{k_t-1}}^{t} \mathbb E\big\|\dot v_{\varepsilon}\big\|_{\mathcal H}^2+ \varepsilon\sum_{l=1}^{k_t}\int_{t_{l-1}}^{t_l}\mathbb E\big\|\dot v_{\varepsilon}\big\|_{\mathcal H}^2\Big]\\
			&\le \varepsilon^2\int_{0}^t \mathbb E\big\|\dot v_{\varepsilon}\big\|_{\mathcal H}^2.
		\end{align*}
		Combining this with (\ref{forfinaline}) yields the result.
	\end{proof}
	
	\begin{crllr}\label{corocc}
		Let $u:[0,+\infty)\to \mathcal H$ be the solution of the gradient flow equation (\ref{gf_eq}) and $v_{\varepsilon}:[0,+\infty)\to \mathcal H$ the solution of the mini-batch descent flow equation (\ref{mb_gf}). Then, for any $\alpha\in[0,1]$, 
		\begin{align*}
			\mathbb E\big\|v_\varepsilon(t)-u(t)\big\|^2_{\mathcal H}\le  2\varepsilon^{\frac{1}{2-\alpha}} \Big(\int_{0}^t \mathbb E\big\|\varepsilon^{\frac{1-\alpha}{2-\alpha}}\dot v_{\varepsilon}(\tau)\big\|_{\mathcal H}^2\Big)^{\frac{1}{2}} \Big(\int_{0}^t \Lambda\big(u(s)\big)\, ds\Big)^{\frac{1}{2}}\quad \forall t\in[0,+\infty).
		\end{align*}
	\end{crllr}

	From previous results, it is possible to say something about the expected asymptotic behavior of the mini-batch descent solution.
	\begin{thrm}\label{Asyncc}
		Let $u:[0,+\infty)\to \mathcal H$ be the solution of the gradient flow equation (\ref{gf_eq}) and $v_{\varepsilon}:[0,+\infty)\to \mathcal H$ the solution of the mini-batch descent flow equation (\ref{mb_gf}).  Suppose that $\Lambda\circ u:[0,+\infty)\to \mathcal H$ is locally integrable and that there exists $\alpha\in(0,1]$ such that, for each $j\in\{1,\dots,m\}$, $\Phi_{\mathcal B_j}$ is locally $\alpha$-H\"older continuous in its effective domain.Then the following statements hold. 
		\begin{itemize}
			\item[(i)]  Suppose  $\dom\partial\Phi_{\mathcal B_j}\subset \dom\Phi$ for all $j\in\{1,\dots,m\}$. Then, for every $\eta>0$ there exists $T>0$ such that 
			\begin{align*}
				\mathbb E\Phi(v_{\varepsilon}(T)) \le \inf_{v\in\mathcal H}\Phi(v) + \eta \quad \text{for all $\varepsilon>0$ small enough.}
			\end{align*}
			
			\item[(ii)] Suppose $\Phi:\mathcal H\to\mathbb R\cup\{+\infty\}$ is inf-compact. Then there exists $u^*\in\argmin_{v\in\mathcal H}\Phi(v)$ such that for every $\eta>0$ there exists $T>0$ such that
			\begin{align*}
				\mathbb E\big\|v_\varepsilon(T) - u^*\big\|_{\mathcal H}^2\le \eta \quad \text{for all $\varepsilon>0$ small enough.}
			\end{align*}
		\end{itemize}
	\end{thrm}
	\begin{proof}
		Let $\eta>0$ be given. 
		\begin{itemize}
			\item[(i)] By \cite[Proposition 17.2.7]{SoBV}, there exists $T>0$ such that $\Phi(u(T))\le \inf_{v\in\mathcal H}\Phi(v) + \eta/2$. Since by Corollary \ref{corobound}, $\{v_{\varepsilon}(T)\}_{\varepsilon>0}$ is bounded, there exists $L_T>0$ such that 
			\begin{align}\label{genth}
				|\Phi(v_\varepsilon(T)) - \Phi(u(T))|\le L_T \big\| v_{\varepsilon}(T) - u(T) \big\|_{\mathcal H}^{\alpha}\quad \forall\varepsilon>0.
			\end{align}
			By Lemma \ref{lemmaintervals0} and Corollary \ref{corocc}, there exists $c_T>0$ such that
			$\mathbb E\big\| v_{\varepsilon}(T)-u(T)\big\|_{\mathcal H}^2\le c_T\varepsilon^{\frac{1}{2-\alpha}}$ for all $\varepsilon>0$. By H\"older's inequality, for all $\varepsilon>0$,
			\begin{align}\label{genth2}
				\mathbb E\big\| v_{\varepsilon}(T) - u(T) \big\|_{\mathcal H}^{\alpha}\le \Big(\mathbb E\big\| v_{\varepsilon}(T)-u(T)\big\|_{\mathcal H}^2\Big)^{\frac{\alpha}{2}}\le c_T^{\frac{\alpha}{2}} \varepsilon^{\frac{\alpha}{2(2-\alpha)}}.
			\end{align}
			Finally, from (\ref{genth}) and (\ref{genth2}),
			\begin{align*}
				\mathbb E\Phi(v_{\varepsilon}(T)) \le \Phi(u(T)) + \mathbb E|\Phi(v_{\varepsilon}(T))-\Phi(u(T))| \le  \inf_{v\in\mathcal H}\Phi(v) + \frac{\eta}{2} + L_Tc_T^{\frac{\alpha}{2}}\varepsilon^{\frac{\alpha}{2(2-\alpha)}},
			\end{align*}
			for all $\varepsilon>0$. It is then enough to take $\varepsilon$ such that $L_Tc_T^{\frac{\alpha}{2}}\varepsilon^{\frac{\alpha}{2(2-\alpha)}}<\eta/2$.

			\item[(ii)] As $\Phi$ is inf-compact and lower semicontinuous, $\argmin_{v\in\mathcal H}\Phi \neq \emptyset$. By \cite[Corollary 17.2.1]{SoBV}, there exists $u^*\in\argmin_{v\in\mathcal H}\Phi$ such that $u(t)\rightharpoonup u^*$ weakly in $\mathcal H$ as $t\to+\infty$. 
			Since on the bounded subsets of the lower level sets of $\Phi$, weak and strong convergence coincide,  $u(t)\to u^*$ strongly in $\mathcal H$ as $t\to+\infty$. Therefore, there exists $T>0$ such that $\big\| u(T) - u^*\big\|_{\mathcal H}\le \sqrt{\eta}/2$. By Lemma \ref{lemmaintervals0} and Corollary \ref{corocc}, there exists $\varepsilon_0>0$ such that $\mathbb E\big\| v_{\varepsilon}(T) - u(T)\big\|_{\mathcal H}^2\le \eta/4$ for all $\varepsilon\in(0,\varepsilon_0)$.
			Therefore, 
			\begin{align*}
				\mathbb E\big\| v_{\varepsilon}(T) - u^*\big\|_{\mathcal H}^2\le 2  	\mathbb E\big\| v_{\varepsilon}(T) - u(T)\big\|_{\mathcal H}^2 + 2 	\mathbb E\big\| u(T) - u^*\big\|_{\mathcal H}^2\le \eta\quad \forall \varepsilon\in(0,\varepsilon_0).
			\end{align*}
		\end{itemize}
		
	\end{proof}

	\subsection{Random minimizing movement}
	It follows from \cite[Theorem 17.2.2]{SoBV} that the right derivative $d^+u/dt$ of the solution of gradient flow equation (\ref{gf_eq}) exists everywhere.  For each $k\in\mathbb N$, set
	\begin{align*}
		\omega^+_{k}:= \Big\| \frac{u(t_k)-u(t_{k-1})}{\varepsilon} - \frac{d^+ u}{dt}(t_{k})\Big\|_{\mathcal H}^2.
	\end{align*}
	Observe that the sequence $\{\omega_k^+\}_{k\in\mathbb N}$ remains bounded by $4\big\|\partial \Phi(u_0)^\circ\big\|_{\mathcal H}^2$.

	\begin{lmm}\label{lemnbcar}
		Let $u:[0,+\infty)\to\mathcal H$ be the solution of the gradient flow equation (\ref{gf_eq}). Then, 
		\begin{align*}
			\mathbb E\Big\|\frac{u(t_{k})-u(t_{k-1})}{\varepsilon} +  \xi_{j_k}(u(t_{k}))\Big\|_{\mathcal H}^2 \le 2\omega_k^+ +  2\Lambda\big(u(t_k)\big)\quad \forall k\in\mathbb N.
		\end{align*}
	\end{lmm}
	\begin{proof}
		Since, by \cite[Theorem 17.2.2]{SoBV}, $d^+u/dt = \partial\Phi(u)^\circ$ everywhere in $[0,+\infty)$; for all $k\in\mathbb N$,
		\begin{align*}
			\Big\|\frac{u(t_{k})-u(t_{k-1})}{\varepsilon} +\xi_{j_k}(u(t_{k}))\Big\|_{\mathcal H}^2&\le 2	\Big\|\frac{u(t_{k})-u(t_{k-1})}{\varepsilon}-\frac{d^+u}{dt}(t_k)\Big\|_{\mathcal H}^2+ 2	\Big\| - \partial\Phi(u(t_{k}))^\circ+ \xi_{j_k}(u(t_{k}))\Big\|_{\mathcal H}^2.
		\end{align*}
		The result follows taking expectation on both sides.
	\end{proof}

	\begin{lmm}\label{lem1rmm}
		Let $u:[0,+\infty)\to\mathcal H$ be the solution of the gradient flow equation (\ref{gf_eq}) and $\{w_k\}_{k\in\mathbb N}$ the proximal sequence (\ref{gf_mm0}).  Then, for all $k\in\mathbb N$,
		\begin{align*}
			\mathbb E\big\|w_k-u(t_k)\big\|_{\mathcal H}^2  \le  2\varepsilon\Big(\sum_{l=1}^k\varepsilon \omega_k^++\sum_{l=1}^k\varepsilon\Lambda\big(u(t_l)\big)  + \sum_{l=1}^{k} \sqrt{\omega_{l}^+}\,\sqrt{\mathbb E\big\|w_{l-1} - u(t_{k-1})\big\|^2_{\mathcal H}}\Big)\quad\forall k\in\mathbb N.
		\end{align*}
	\end{lmm}
	\begin{proof}
		It is known that proximal sequence  (\ref{gf_mm0}) satisfies the differential inclusions
		\begin{align*}
			\displaystyle-\frac{w_{l}-w_{l-1}}{\varepsilon} \in \partial \Phi_{\mathcal B_{j_l}}(w_l)\quad\forall l\in\mathbb N,
		\end{align*}
		where $w_0:=u_0$.
		For each $l\in\mathbb N$, define $	f_l:=-\Big[\varepsilon^{-1}\big(u(t_{l})-u(t_{l-1})\big) + \xi_{j_l}(u(t_{l}))\Big].$
		Observe that 
		\begin{align*}
			-f_l-\frac{u(t_{l})-u(t_{l-1})}{\varepsilon} \in \partial \Phi_{\mathcal B_{j_{l}}}(u(t_{l}))\quad \forall l\in\mathbb N.
		\end{align*}
		By monotonicity of the subdifferential, 
		\begin{align}\label{stbsid}
			\Big\langle -\frac{w_{l}-w_{l-1}}{\varepsilon} + f_l+\frac{u(t_{l})-u(t_{l-1})}{\varepsilon}, w_l-u(t_l)\Big\rangle \ge 0 \quad \forall l\in\mathbb N_{}.
		\end{align}
		Set $\mathcal E_0:=0$, and for each $l\in\mathbb N$, denote $\mathcal E_{l}:=w_l - u(t_l)$. Then (\ref{stbsid}) simplifies to $\big \langle \mathcal E_{l} - \mathcal E_{l-1}, \mathcal E_l \big\rangle \le  \varepsilon \langle f_l, \mathcal E_{l}\rangle$ for all $l\in\mathbb N$. 
		Using the identity 
		\begin{align*}
			\big\|a\big\|_{\mathcal H}^2 - \big\|b\big\|_{\mathcal H}^2+\big\|a-b\big\|_{\mathcal H}^2= 	2\langle a-b, a\rangle\quad\forall a,b\in\mathcal H,
		\end{align*}
		we get $\big\|\mathcal E_l\big\|_{\mathcal H}^2 - \big\|\mathcal E_{l-1}\big\|_{\mathcal H}^2 + \big\|\mathcal E_l-\mathcal E_{l-1}\big\|_{\mathcal H}^2  \le 2 \varepsilon \langle \mathcal E_l, f_l \rangle$ for all $l\in\mathbb N$, and hence
		\begin{align*}
			\big\|\mathcal E_l\big\|_{\mathcal H}^2 - \big\|\mathcal E_{l-1}\big\|_{\mathcal H}^2 + \big\|\mathcal E_l-\mathcal E_{l-1}\big\|_{\mathcal H}^2  \le 2 \varepsilon \langle \mathcal E_l, f_l \rangle \le 2 \varepsilon \langle \mathcal E_l-\mathcal E_{l-1}, f_l \rangle + 2 \varepsilon \langle \mathcal E_{l-1}, f_l \rangle \quad \forall l\in\mathbb N.
		\end{align*}
		Employing Fenchel-Young inequality yields
		\begin{align*}
			\big\|\mathcal E_l\big\|_{\mathcal H}^2 - \big\|\mathcal E_{l-1}\big\|_{\mathcal H}^2 + \big\|\mathcal E_l-\mathcal E_{l-1}\big\|_{\mathcal H}^2\le 2\Big(\frac{\big\|\mathcal E_l-\mathcal E_{l-1}\big\|_{\mathcal H}^2}{2} +  \frac{\varepsilon^2 \big\|f_{l}\big\|_{\mathcal H}^2}{2}\Big) + 2 \varepsilon \langle \mathcal E_{l-1}, f_l \rangle,
		\end{align*}
		for all $l\in\mathbb N$.
		Hence, 
		\begin{align*}
			\big\|\mathcal E_l\big\|_{\mathcal H}^2 - \big\|\mathcal E_{l-1}\big\|_{\mathcal H}^2 + \le    \varepsilon^2 \big\|f_{l}\big\|_{\mathcal H}^2 +  2 \varepsilon \langle \mathcal E_{l-1}, f_l \rangle\quad\forall l\in\mathbb N.
		\end{align*}
		Summing up and taking expectation, 
		\begin{align*}
			\mathbb E\big\|\mathcal E_k\big\|_{\mathcal H}^2  \le  \varepsilon^2\sum_{l=1}^k \mathbb E\big\|f_k\big\|_{\mathcal H}^2 + 2\varepsilon\sum_{l=1}^{k}\mathbb E \langle \mathcal E_{k-1}, f_k \rangle\quad\forall k\in\mathbb N.
		\end{align*}
		Now, observe that $\mathbb E\langle \mathcal E_{l-1},f_l\rangle =\langle \mathbb E[\mathcal E_{l-1}], \mathbb E f_l \rangle = \langle \mathbb E[\mathcal E_{l-1}],  d^+u/dt(t_l) - \varepsilon^{-1}\big(u(t_l)-u(t_{l-1})\big)\rangle$ for all $l\in\mathbb N$, and that, by Lemma \ref{lemnbcar}, $\mathbb E\big\| f_l\big\|_{\mathcal H}^2\le2 \omega_l^+ + 2\Lambda\big(u(t_l)\big)$ for all $l\in\mathbb N$. Thus, 
		\begin{align*}
			\mathbb E\big\|\mathcal E_k\big\|_{\mathcal H}^2   \le  2\varepsilon^2\Big(\sum_{l=1}^k \omega_l^++\sum_{l=1}^k\Lambda\big(u(t_l)\big) \Big) + 2\varepsilon\sum_{l=1}^{k} \sqrt{\omega_{l}^+}\,\mathbb E\big\|\mathcal E_{l-1}\big\|_{\mathcal H}\quad\forall k\in\mathbb N.
		\end{align*}
		Applying Cauchy-Schwartz inequality and rearranging yields the desired estimate. 
	\end{proof}
	
	\begin{lmm}\label{lem2rmm}
		Let $u:[0,+\infty)\to\mathcal H$ be the solution of the gradient flow equation (\ref{gf_eq}) and $\{w_k\}_{k\in\mathbb N}$ the proximal sequence (\ref{gf_mm0}). Then, 
		\begin{align*}
			\mathbb E\big\|w_k - u(t_k)\big\|_{\mathcal H}^2\le 4\varepsilon\Big(2t_k\sum_{l=1}^k\omega_{l}^++\sum_{l=1}^{k}\varepsilon\Lambda\big(u(t_l)\big)\Big)\quad \forall k\in\mathbb N.
		\end{align*}
	\end{lmm}
	\begin{proof}
		Let $k\in\mathbb N$ be given.  For each $l\in\mathbb N$, denote $\mathcal E_{l}:=w_l - u(t_l)$. By Lemma \ref{lem1rmm}, for all $l\in\{1,\dots, k\}$, 
		\begin{align*}
			\mathbb E\big\|\mathcal E_l\big\|^2 &\le 2\varepsilon\Big(\sum_{i=1}^l\varepsilon \omega_i^++\sum_{i=1}^l\varepsilon\Lambda\big(u(t_i)\big)  + \sum_{i=1}^{l} \sqrt{\omega_{i}^+}\,\sqrt{\mathbb E\big\|\mathcal E_{i-1}\big\|_{\mathcal H}^2}\Big)\\
			&\le 2\varepsilon\Big(\sum_{i=1}^k\varepsilon \omega_i^++\sum_{i=1}^k\varepsilon\Lambda\big(u(t_i)\big)  + \sum_{i=2}^{k} \sqrt{\omega_{i}^+}\,\sqrt{\mathbb E\big\|\mathcal E_{i-1}\big\|_{\mathcal H}^2}\Big)\\
			&\le 2\varepsilon\Big(\sum_{i=1}^k\varepsilon \omega_i^++\sum_{i=1}^k\varepsilon\Lambda\big(u(t_i)\big)  + \Big[\max_{\{1,\dots,k\}}\mathbb E\big\|\mathcal E_{i}\big\|_{\mathcal H}^2\Big]^{\frac{1}{2}} \sum_{i=2}^{k} \sqrt{\omega_{i}^+}\Big).
		\end{align*}
		Let $x_k:=\displaystyle\Big[\max_{\{1,\dots,k\}}\mathbb E\big\|\mathcal E_i\big\|_{\mathcal H}^2\Big]^{\frac{1}{2}}$,  $a_k:=\displaystyle\varepsilon\sum_{i=2}^{k} \sqrt{\omega_{i}^+}$, and
		\begin{align*}
			b_k:= \sqrt{2\varepsilon}\Big(\sum_{i=1}^k\varepsilon \omega_i^++4\varepsilon\sum_{i=1}^k\varepsilon\Lambda\big(u(t_i)\big)  \Big)^{\frac{1}{2}}.
		\end{align*}
		Then, $x_k^2 \le 2a_k x_k+ b_k^2$; this implies $x_k \le 2a_k + b_k$, that is, 
		\begin{align*}
			\max_{\{1,\dots,k\}}\mathbb E\big\|\mathcal E_i\big\|_{\mathcal H}^2&\le \Big(2\varepsilon\sum_{i=2}^{k} \sqrt{\omega_{i}^+} +\sqrt{2\varepsilon}\Big(\sum_{i=1}^k\varepsilon \omega_i^++4\varepsilon\sum_{i=1}^k\varepsilon\Lambda\big(u(t_i)\big)  \Big)^{\frac{1}{2}}\Big)^2\\
			&\le 8\varepsilon^2\Big(\sum_{i=2}^{k} \sqrt{\omega_{i}^+}\Big)^2 + 4\varepsilon\Big(\sum_{i=1}^k\varepsilon \omega_i^++4\varepsilon\sum_{i=1}^k\varepsilon\Lambda\big(u(t_i)\big)  \Big).
		\end{align*}
		Applying Cauchy-Schwartz inequality, and rearranging
		\begin{align*}
			\max_{\{1,\dots,k\}}\mathbb E\big\|\mathcal E_i\big\|_{\mathcal H}^2\le 4\varepsilon^2\Big(2(k-1)\sum_{i=2}^k\omega_i^+ +\sum_{i=1}^k\omega_i^++ \sum_{i=1}^{k}\Lambda\big(u(t_i)\big)\Big)\le 4\varepsilon^2\Big(2k\sum_{i=1}^k\omega_{i}^++\sum_{i=1}^{k}\Lambda\big(u(t_i)\big)\Big).
		\end{align*}
	\end{proof}

	\begin{thrm}\label{thrrmm}
		Let $u:[0,+\infty)\to\mathcal H$ be the solution of the gradient flow equation (\ref{gf_eq}) and $w_{\varepsilon}:[0,+\infty)\to\mathcal H$ the function generated by the proximal sequence (\ref{gf_mm0}). Let $T>0$, and set $N:=\lfloor T/\varepsilon \rfloor$. Then, 
		\begin{align*}
			\mathbb E\big\|w_\varepsilon(t) - u(t)\big\|_{\mathcal H}^2\le 4\varepsilon\Big(2T\sum_{i=1}^{N}\omega_{i}^++\sum_{i=1}^{N}\varepsilon\Lambda\big(u(t_i)\big)\Big)+2\varepsilon^2\big\|\partial\Phi(u_0)^\circ\big\|_{\mathcal H}^2.\quad \forall t\in[0,T].
		\end{align*}
	\end{thrm}
	\begin{proof}
		Let $k\in\mathbb N$. For $t\in[t_{k-1},t_{k})$, $w_{\varepsilon}(t) = w_{k}$, and hence 
		\begin{align*}
			\big\|w_\varepsilon(t) - u(t)\big\|_{\mathcal H}^2\le 2 \big\|w_k - u(t_k)\big\|_{\mathcal H}^2 + 2\big\|u(t_k)- u(t)\big\|_{\mathcal H}^2\le 2 \big\|w_k - u(t_k)\big\|_{\mathcal H}^2 +2\varepsilon^2\big\|\partial\Phi(u_0)^\circ\big\|_{\mathcal H}^2.
		\end{align*}
		The result follows then from Lemma \ref{lem2rmm}.
	\end{proof}
	\begin{crllr}\label{corfin}
		Let $u:[0,+\infty)\to\mathcal H$ be the solution of the gradient flow equation (\ref{gf_eq}) and $w_{\varepsilon}:[0,+\infty)\to\mathcal H$ the function generated by the proximal sequence (\ref{gf_mm0}). Let $T>0$. Assume that $\Lambda\circ u$ is bounded in $[0,T]$. If $u$ belongs to $C^1\big([0,T];\mathcal H\big)$, then 
		\begin{align*}
			\sup_{t\in[0,T]}\mathbb E\big\|w_\varepsilon(t) - u(t)\big\|_{\mathcal H}^2\longrightarrow 0\quad \text{as}\quad \varepsilon\longrightarrow0^+.
		\end{align*}
		Moreover, if there exists $\alpha\in(0,1]$ such that $\dot u$ is $\alpha$-H\"older continuous in $[0,T]$, then there exists $c_T>0$ such that 
		\begin{align*}
			\sup_{t\in[0,T]}\mathbb E\big\|w_\varepsilon(t) - u(t)\big\|_{\mathcal H}^2\le c_T\big(\varepsilon^{2\alpha}+\varepsilon^2+\varepsilon\big)\quad \forall\varepsilon>0.
		\end{align*}
	\end{crllr}
	\begin{proof}
		Since $\dot u$ is continuous in the compact set $[0,T]$, it is uniformly continuous over it; hence there exists $\eta:[0,T]\to[0,+\infty)$ such that
		\begin{align*}
			\big\|\dot u(t)-\dot u(s)\big\|_{\mathcal H}\le \eta\big(|t-s|\big) \quad \forall s,t\in[0,T].
		\end{align*}
		Set $N:=\lfloor T/\varepsilon \rfloor$, for all $k\in\{1,\dots,N\}$,
		\begin{align*}
			\sqrt{\omega_k^+}= \Big\| \frac{u(t_k)-u(t_{k-1})}{\varepsilon} - \frac{d^+ u}{dt}(t_{k})\Big\|_{\mathcal H}=\Big\|\frac{1}{\varepsilon}\int_{t_{k-1}}^{t_k}\big(\dot u(s)-\dot u(t_k)\big)\,ds\Big\|_{\mathcal H}\le \frac{1}{\varepsilon}\int_{t_{k-1}}^{t_k}\big\|\dot u(s) - \dot u(t_k)\big\|_{\mathcal H}\,ds.
		\end{align*}
		This implies $\omega_k^+\le \eta(\varepsilon)^2$ for all $k\in\{1,\dots,N\}$. Now, from Theorem \ref{thrrmm}, 
		\begin{align*}
			\mathbb E\big\|w_\varepsilon(t) - u(t)\big\|_{\mathcal H}^2\le 8T^2\eta(\varepsilon)^2+\varepsilon\sum_{i=1}^{N}\varepsilon\Lambda\big(u(t_i)\big)+2\varepsilon^2\big\|\partial\Phi(u_0)^\circ\big\|_{\mathcal H}^2\quad \forall t\in[0,T].
		\end{align*}
		Let $M:=\sup_{[0,T]}|\Lambda(u(t))|$. Then, 
		\begin{align}\label{estimatermm}
			\sup_{t\in[0,T]}\mathbb E\big\|w_\varepsilon(t) - u(t)\big\|_{\mathcal H}^2\le 8T^2\eta(\varepsilon)^2+M\varepsilon+2\varepsilon^2\big\|\partial\Phi(u_0)^\circ\big\|_{\mathcal H}^2.
		\end{align}
		Letting $\varepsilon\to0^+$, yields the first part the result. Now, if $\dot u$ is $\alpha$-H\"older continuous, the modulus of continuity of $\dot u$ over $[0,T]$ can be taken as $\eta(t)=L_Tt^\alpha$.  Then, (\ref{estimatermm}) becomes
		\begin{align*}
			\sup_{t\in[0,T]}\mathbb E\big\|w_\varepsilon(t) - u(t)\big\|_{\mathcal H}^2\le 8T^2\varepsilon^{2\alpha}+M\varepsilon+2\varepsilon^2\big\|\partial\Phi(u_0)^\circ\big\|_{\mathcal H}^2.
		\end{align*}
	\end{proof}

\section{Applications to optimization and partial differential equations}\label{sec:applications}
In this section, we will show how our main results can be applied to problems related to optimization and partial differential equations. We will also provide some numerical examples to illustrate our results.

	\subsection{Constrained optimization}\label{projdyn}
	Let $\mathcal H$ be a Hilbert space and $C$ a closed convex bounded subset of $\mathcal H$.  Let $\Psi_1,\dots,\Psi_{n}:\mathcal H\to\mathbb R$ be convex continuously differentiable functions. Let $\pi_1,\dots,\pi_{m}$ be positive numbers such that $\pi_1+\dots+\pi_m=1$. Consider the average potential $\Psi:\mathcal H\to\mathbb R$ given by
	\begin{align*}
		\Psi(u):= \sum_{j=1}^m\pi_{j}\Psi_{j}(u).
	\end{align*} 
	In this subsection, we are interested in the following constrained optimization problem: \begin{align}\label{consprob}
		\min_{u\in C} \Psi(u).
	\end{align}
	It is well know that if $u^*\in C$ is a minimizer of problem (\ref{consprob}), then $\langle \nabla \Psi(u^*), v\rangle\ge 0$ for all $v\in T_C(u^*)$, where $T_{C}:\mathcal H\twoheadrightarrow \mathcal H$ is the tangent cone mapping of $C$.
	From \cite[Proposition 17.2.12]{SoBV}, given an initial datum $u_{0}\in C$, the gradient flow associated with the problem (\ref{consprob}) is given by 
	\begin{align}\label{projgf}
		\left\{ \begin{array}{l} \dot u(t) = - \proj_{T_C(u(t))}\big(\nabla \Psi (u(t))\big), \\  u(0)=u_0. \end{array} \right.
	\end{align}
	Following the procedure described in Section \ref{sec:problem_formulation}, it is possible to construct a continuous function, depending on a parameter $\varepsilon>0$,  $v_{\varepsilon}:[0,+\infty)\to\mathcal H$ such that $v_{\varepsilon}(0) = u_0$, and for each $k\in\mathbb{N}$,
	\begin{align}\label{projgfmb}
		0 = {\dot v_{\varepsilon}}(t) + \proj_{T_C(u(t))}\big(\nabla \Psi_{j_k}(v_{\varepsilon}(t))\big)\quad \text{for a.e. $t\in[(k-1)\varepsilon, k\varepsilon)$}.
	\end{align}
	Here, $\{j_{l}\}_{k\in\mathbb N}$ is a sequence of random variables taking value $j\in\{1,\dots,m\}$ with probability $\pi_j$.  
	
	\subsubsection{Convergence and asymptotic behavior}
	In order to quantify the variance induced by the replacement of gradients over time, consider the function $\Gamma:\mathcal H\to\mathbb R$ given by
	\begin{align*}
		\Gamma(u):=\sum_{j=1}^m \pi_{j} \big\|\nabla \Psi_{j}(u)-\nabla\Psi(u)\big\|_{\mathcal H}^2.
	\end{align*}
	Observe that $	\mathbb E\big[\nabla\Psi_{j_k}\big] = \nabla\Psi$ and $\Var\big[\nabla\Psi_{j_k}\big] = \Gamma$ for all $k\in\mathbb N$.  Also, since $\Gamma$ is continuous, $\Gamma\circ u:[0,+\infty)\to\mathbb R$ is locally bounded. 
	\begin{thrm}\label{th:theorem_41}
		There exists a unique locally absolutely continuous function $v_{\varepsilon}:[0,+\infty)\to \mathcal H$ satisfying $v_{\varepsilon}(0)=u_0$ and (\ref{projgfmb}). Moreover, $v_{\varepsilon}$ is locally Lipschitz, and the  following statements hold. 
		\begin{itemize}
			\item[(i)] For every $T>0$ there exists $c_{T}>0$ such that 
			\begin{align*}
				\sup_{t\in[0,T]}\mathbb E\big\| v_{\varepsilon}(t)-u(t)\big\|_{\mathcal H}^2\le c_{T}\varepsilon \int_{0}^T\Gamma\big(u(s)\big)\,ds \quad\forall \varepsilon>0.
			\end{align*}
			
			\item[(ii)] For every $\eta>0$ there exists $T>0$ such that 
			\begin{align*}
				\mathbb E\Phi(v_{\varepsilon}(T)) \le \inf_{v\in\mathcal H}\Phi(v) + \eta \quad \text{for all $\varepsilon>0$ small enough.}
			\end{align*}
			
			\item[(iii)] There exists $u^*\in\argmin_{v\in C}\Psi(v)$ such that for every $\eta>0$ there exists $T>0$ satisfying
			\begin{align*}
				\mathbb E\big\|v_\varepsilon(T) - u^*\big\|_{\mathcal H}^2\le \eta \quad \text{for all $\varepsilon>0$ small enough.}
			\end{align*}
		\end{itemize}
	\end{thrm}
	\begin{proof}
		For each $j\in\{1,\dots, m\}$, define $\Phi_{j}:=\Psi_j + \delta_{C}$, where $\delta_C:\mathcal H\to\mathbb R\cup\{+\infty\}$ is the indicator function of $C$. By Moreau–Rockafellar subdifferential additivity rule, $\partial\Phi_j = \nabla\Psi_j + N_{C}(u)$ for each $j\in\{1,\dots,m\}$. Let $\Phi:=\Psi + \delta_{C}$; we see that $\partial \Phi = \sum_{j=1}^m\pi_j\partial \Phi_{j}$. For each $u\in C$, denote by $\eta^*(u)$ the unique element in $N_{C}(u)$ such that $\partial \Phi(u)^\circ = \nabla \Psi(u) + \eta^*(u)$. For each $j\in\{1,\dots, m\}$, define $\xi_{j}:C\to \mathcal H$ by $\xi_j(u) := \nabla\Psi_j(u)+ \eta^*(u)$.  We see that  $\sum_{j=1}^m\pi_j\xi_j(u) = \nabla\Psi(u) + \eta^*(u)=\partial\Phi(u)^\circ$ for all $u\in C$. Consider now, the function $\Lambda:C\to \mathbb R$ in (\ref{Lambda}) based on the previous decomposition of the minimal norm subdifferential; we see then that
		\begin{align*}
			\Lambda(u) = \sum_{j=1}^m\pi_j\big\|\xi_j(u)-\partial\Phi(u)^\circ\big\|_{\mathcal H}^2=\sum_{j=1}^m \pi_{j} \big\|\nabla \Psi_{j}(u)-\nabla\Psi(u)\big\|_{\mathcal H}^2 = \Gamma(u).
		\end{align*}
		Now, observe that since, for each $j\in\{1,\dots,m\}$, $\Psi_j$ is locally Lipschitz, so is $\Phi_j$ over $C$.  We can then employ Theorems \ref{existencethrm} and \ref{resultcontinuouscase} to conclude the result.
	\end{proof}

	\subsubsection{Illustrative numerical example}
	To illustrate Theorem \ref{th:theorem_41}, let $u=(u^1,u^2)\in\R^2$ and $u_d=(u^1_d,u^2_d)$, we consider the quadratic programming problem given by
 \begin{align}\label{eq:ex_num_1_1}
     \min_{u\in \mathcal{C}} \left\{ \Psi(u) =\left(2|u^1-u_d^1|+3|u^2-y_d|^2 -2u^1 -3u^2   \right)\right\},
  \end{align}
where the feasible set $\mathcal{C}$ is a convex set, defined by the $u=(u^1,u^2)$ such that 
\begin{align*}
       \mathcal{C}=\{u=(u^1,u^2)\in \R^2\,:\,5u^1+3u^2\leq 120,\, 4u^1+6u^2\leq 150, \,u^1-2u^2 \leq 0,\, u^1\geq 7,\, u^2\leq 15\}. 
\end{align*}
Now, we consider the gradient flow associated with \eqref{eq:ex_num_1_1}, which corresponds to the system \eqref{projgf}. To implement it, we consider the dynamic in a fixed time interval $[0,T]$ with $T=10$ and $h=0.01$ as time steps. We fix $u_d=(10,20)$.We look at three initial points: $(8,4)$, $(13,8)$, and $(20,14)$, which are at a corner, inside, and on the boundary of the feasible set, respectively. We use an explicit Euler method to compute the gradient flow. The projection onto the feasible set $\mathcal{C}$ is performed through an iterative process in which, after each gradient descent step, any violated constraints are sequentially enforced by adjusting the solution in the direction opposite to the normal vectors of the feasible set boundary. The solution of this system is illustrated in \Cref{fig:constrain}. 

\begin{figure}
\centering
\subfloat[]{
\includegraphics[width=0.33\textwidth]{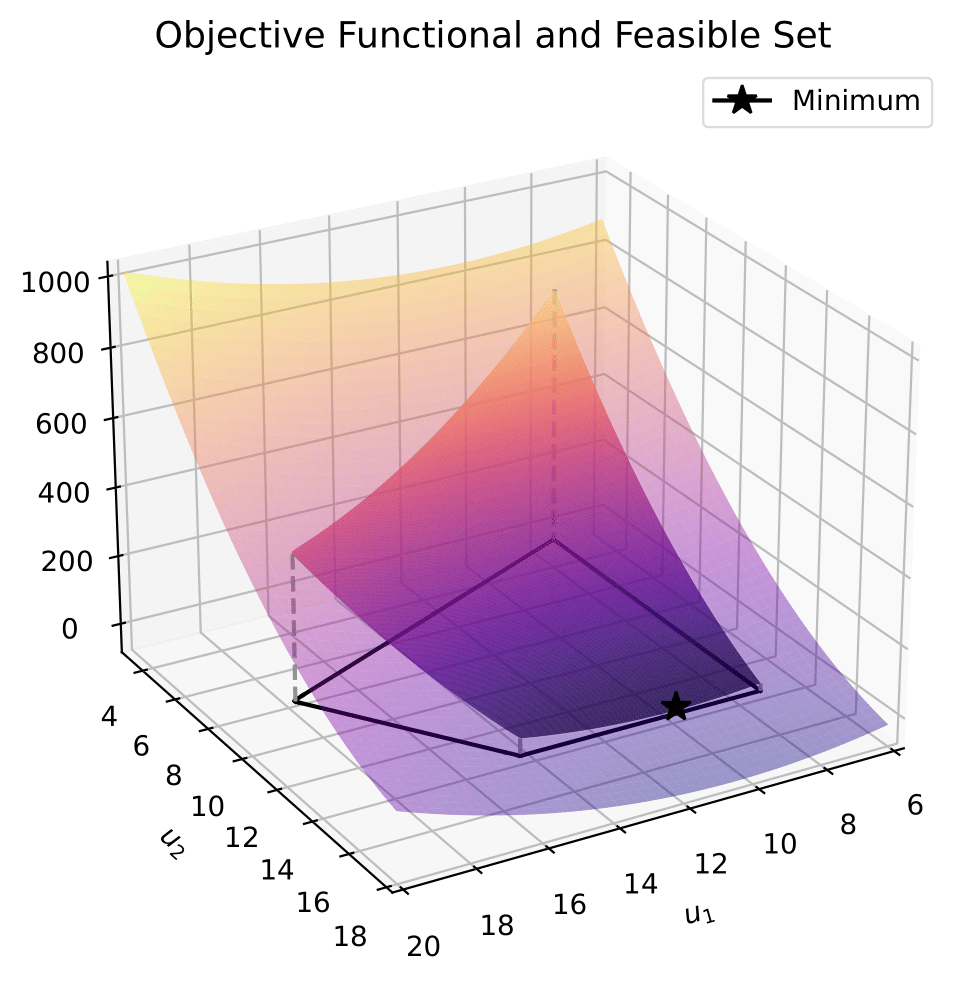}}
\subfloat[]{
\includegraphics[width=0.33\textwidth]{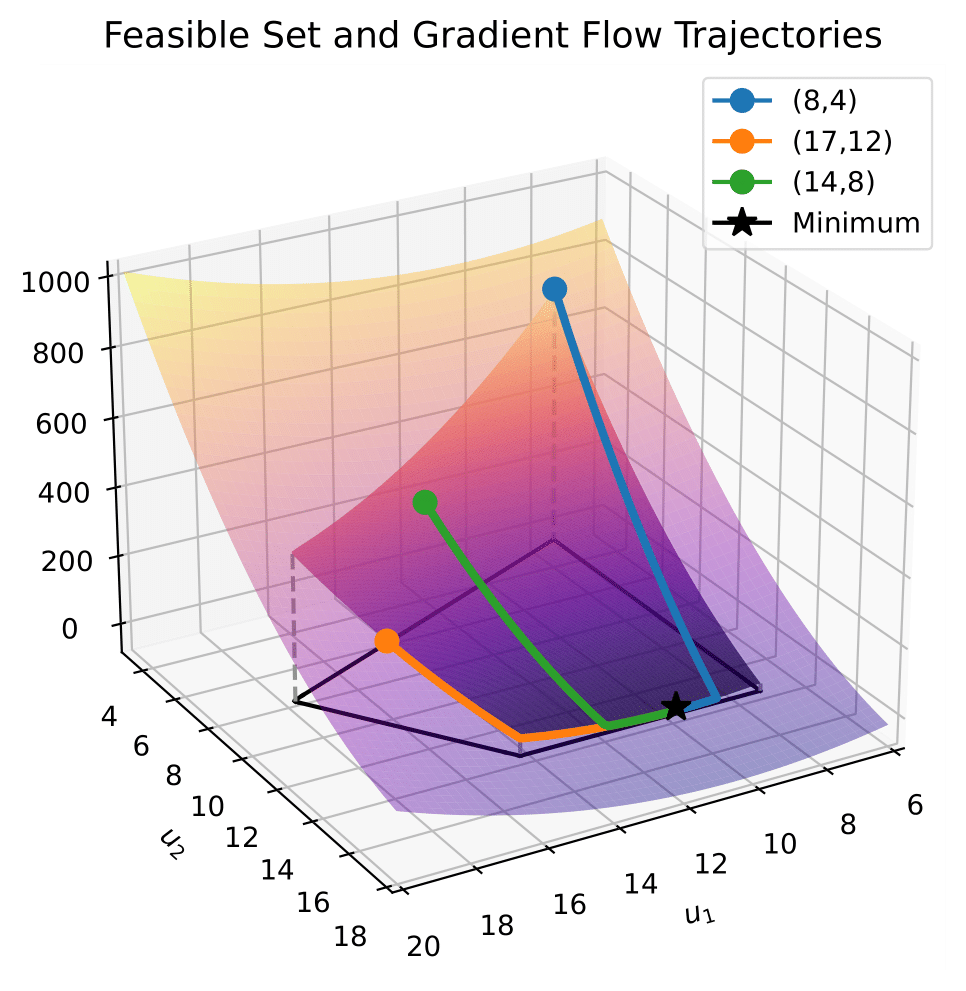}}
\subfloat[ ]{
\includegraphics[width=0.3\textwidth]{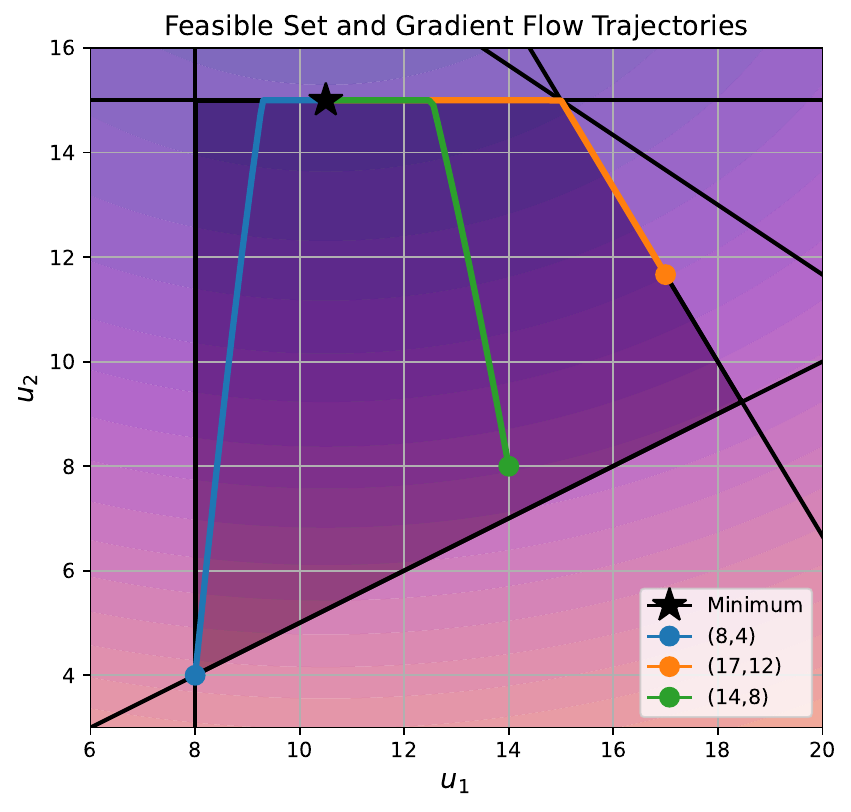}}
\caption{(A) Illustration of the cost function for $(u_1, u_2) \in [6, 10] \times [4, 18]$. We can also observe the feasible set (and its projection onto $\mathbb{R}^2$), illustrated with a darker color. The star denotes the minimum of the functional in the feasible set. (B) Trajectories of the gradient flow for different initial points. (C) Projection of the feasible set onto the plane, along with the trajectories of the gradient flow. The contour lines of the functional are also illustrated.}\label{fig:constrain}
\end{figure}

In \Cref{fig:constrain}, we observe that starting the dynamics from point $(8,4)$, placed at a corner, the dynamics evolve inside the feasible set until it reaches the boundary and then move to the minima of the functional $\Phi$. A similar behavior is seen in the dynamics with initial point $(13,8)$. For the point $(20,14)$, which is placed on the boundary of the feasible set, the trajectory stays on the boundary until it reaches the minimum.

For each $j\in\{1,2,3\}$ let us consider the functionals $\Psi_j$ given by 
\begin{align*}
    \Psi_1(u)=\pi_1\Psi,\quad \Psi_2(u)=\pi_2(4|u^1-u_d^1|-4u^1),\quad \Psi_3(u)= \pi_3(6|u^2-y_d|^2 -6u^2),
\end{align*}
where $\pi_1=1/2$ and $\pi_2=\pi_3=1/4$. Observe that $\Phi(u)=\sum_{j=1}^3\pi_j\Phi_j(u)$. Then, choosing $\varepsilon=0.04$, we can introduce the mini-batch gradient descend dynamics defined by the system \eqref{projgfmb} that evolves in the time interval $[0,T]$. To solve numerically this problem, we consider the same setting as the gradient flow.

In \Cref{fig:constrain_MB}, several realizations are considered to estimate the average mini-batch descent flow. We observe that the estimated average is close to the gradient flow trajectories. Additionally, due to the construction of the sub-functionals, we observe that $\Psi_2$ only depends on $u^1$ and $\Psi_3$ only depends on $u^2$. Therefore, when the second or third sub-functional is randomly chosen, the mini-batch trajectories move only in the $u^1$ or $u^2$ directions, respectively.
 
\begin{figure}
\centering
\subfloat[]{
\includegraphics[width=0.33\textwidth]{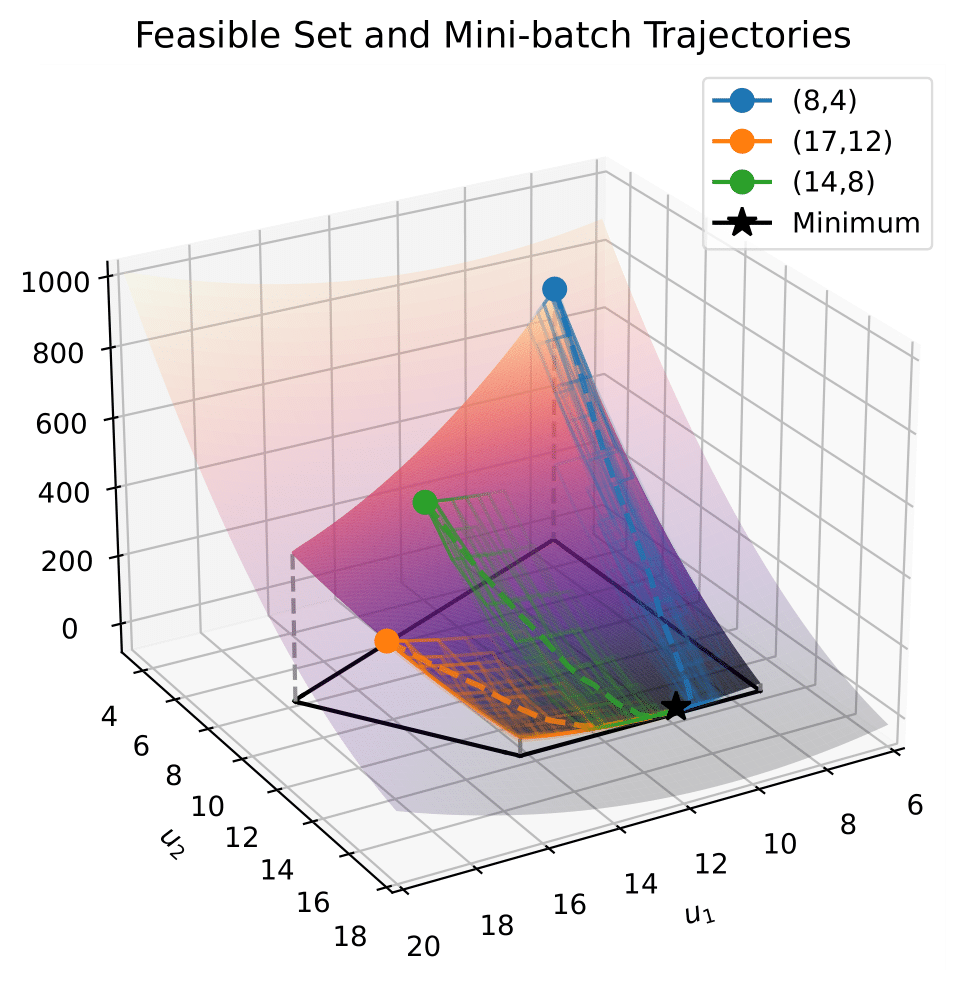}}
\subfloat[]{
\includegraphics[width=0.33\textwidth]{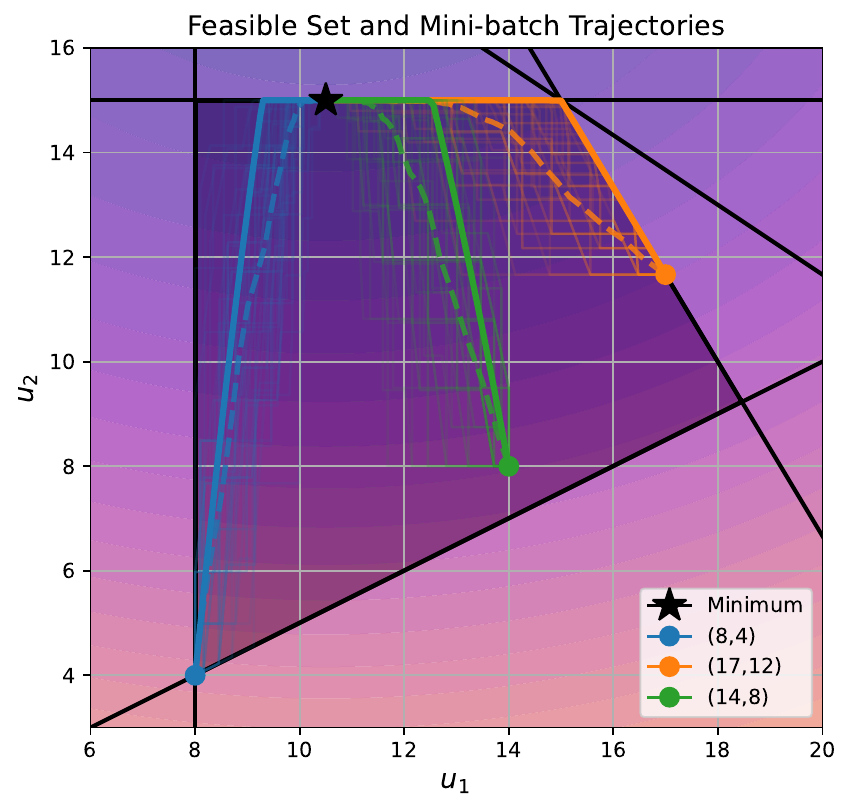}}
\caption{(A) Thin lines show different realizations of the mini-batch descent flow. The dashed line illustrates the average of the different realizations to approximate the average mini-batch descent flow. The solid line shows the previously calculated gradient flow. (B) We illustrate the projection onto the $\R^2$ plane of the trajectories mentioned in (A).}\label{fig:constrain_MB}
\end{figure}

Finally, to corroborate Theorem \ref{th:theorem_41}, we denote by $K$ the number of batch switching in the interval $[0,T]$, of the mini-batch  flow, that is, $K=1/\varepsilon$. Then, taking $K\to \infty$
 (equivalent $\varepsilon\to0$) we can observe the linear convergence (guaranteed by Theorem \ref{th:theorem_41}) in \cref{fig:constrain_convergence}.

\begin{figure}
\centering
\includegraphics[width=0.4\textwidth]{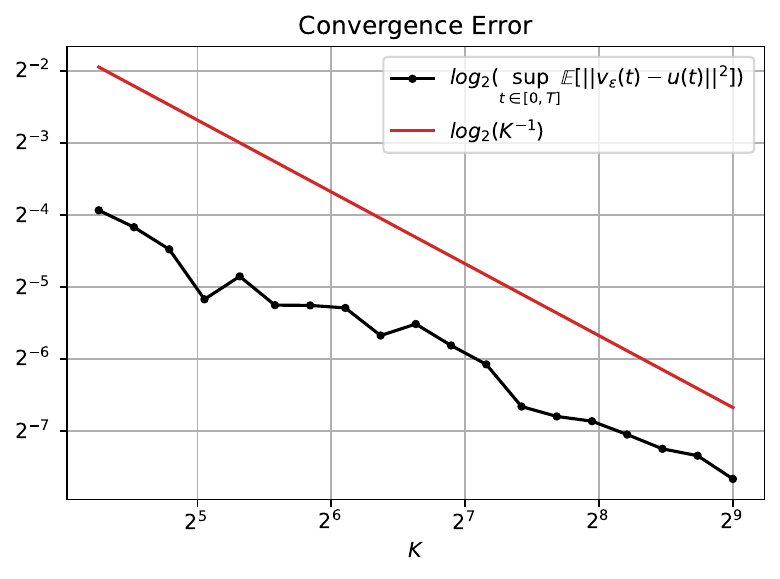}
\caption{We illustrate the convergence as $K\to \infty$ (equivalent to $\varepsilon\to0$) of the mini-batch descent flow.}\label{fig:constrain_convergence}
\end{figure}

	\subsection{Sparse inversion}\label{Sparseinv}
	Let $r,\,d$ positive integers and $\lambda\ge 0$. Consider the matrices $A\in\mathbb R^{r\times d}$ and $b\in \mathbb R^r$. We are concerned with the following \textit{sparse inversion problem}.
	\begin{align}\label{sparseinversion}
		\min_{u\in \mathbb R^d} \left\lbrace \frac{1}{2}\big\| Au - b \big\|_{2}^2 + \lambda\big\| u \big\|_{1} \right\rbrace.
	\end{align}
	The idea behind problem (\ref{Sparseinv}) is to find a ``sparse" $u\in\mathbb R^d$ such that $Au\approx b$.  
	
	Let $\Phi:\mathbb R^d\to\mathbb R$ be the objective functional associated with problem (\ref{Sparseinv}), that is,  \[\Phi(u):=\frac{1}{2}\big\| Au - b \big\|_{2}^2 + \lambda\big\| u \big\|_{1}.\]
	We can readily see that  $\partial\Phi(u)=A^\top\big(Au-b\big)+\Theta(u)$, where $\Theta:\mathbb R^d\twoheadrightarrow \mathbb R^d$ is the set-valued mapping given by
	\begin{align*}
		\Theta(u) := \left\lbrace \eta\in\mathbb R^d:\, \eta_i(x) \in\begin{cases}
			\{1\} & \text{if} \, u_i(x)>0 \\
			[-1,1] & \text{if} \, u_i(x) = 0\\
			\{-1\} & \text{if} \, u_i(x) < 0
		\end{cases} \,\,\, \text{for all $i\in\{1,\dots,d\}$} \right\rbrace .
	\end{align*}

	Given an initial datum $u_0\in \mathbb R^d$, the gradient flow associated with problem (\ref{Sparseinv}), from now on called \textit{sparse inversion flow}, is given by
	\begin{align}\label{eq:gf_sparse}
		\left\{ \begin{array}{l} -\dot u(t)\in A^\top\big(Au(t)-b\big)+\lambda\Theta(u(t)) , \\  u(0)=u_0. \end{array} \right.
	\end{align}
	
	Following the procedure described in Section \ref{sec:problem_formulation}, it is possible to construct a continuous function, depending on a parameter $\varepsilon>0$,  $v_{\varepsilon}:[0,+\infty)\to\mathcal H$ such that $v_{\varepsilon}(0) = u_0$, and for each $k\in\mathbb{N}$,
	\begin{align}\label{inversionmb1}
		&-{\dot v_{\varepsilon}}(t) = \frac{1}{\pi_1}A^\top\big(Av_{\varepsilon}(t)-b\big) \quad \text{ if $j_k=1$},\\
		&-{\dot v_{\varepsilon}}(t) \in \frac{\lambda}{\pi_2}\Theta\big(v_\varepsilon(t)\big) \quad \text{if $j_k=2$}.\label{inversionmb2}
	\end{align}
	for a.e. $t\in[(k-1)\varepsilon, k\varepsilon)$. Here, $\{j_{l}\}_{k\in\mathbb N}$ is a sequence of random variables taking values $j_k=1$ with probability $\pi_1$, and value $j_k=2$ with probability $\pi_2$.  We see then that for $k\in\mathbb N$, 
	\begin{align}\label{eq:explicit_sol_sparse}
		v_{\varepsilon}(t) = \begin{cases}
			\displaystyle e^{-\pi_1^{-1}A^\top A t} v_{\varepsilon}(t_{k-1}) + \pi^{-1}_1\int_{t_{k-1}}^{t} e^{-\pi_1^{-1}A^\top A (t-s)} A^\top b\, ds& \text{if}\ j_k=1 \\
			\big(\sgn v_{\varepsilon}^i(t_{k-1})\max\{|v^i_\varepsilon(t_{k-1})|-\pi_2^{-1}\lambda t,0\}\big)_{i=1}^d & \text{if}\ j_k=2
		\end{cases}\quad \forall t\in[t_{k-1},t_k].
	\end{align}
	At step $k\in\mathbb N$, if $j_k=1$, there is a closed form solution of (\ref{inversionmb1}); on the other hand, if $j_k=2$, the solution of  (\ref{inversionmb2}) possesses a reasonable expression and can be computed by linearly reducing the vector components of the preceding step to zero.
	\subsubsection{Convergence and asymptotic behavior}
	In order to quantify the variance induced by replacing the gradients over a time interval, consider the function $\Gamma:\mathcal H\to\mathbb R$ given by
	\begin{align*}
		\Gamma(u):=\frac{\pi_2^2}{\pi_1}\big\| A^\top \big(A u - b\big)\big\|_2^2 + \frac{\pi_1^2}{\pi_2} (\lambda d)^2.
	\end{align*}
	It is not hard to see that $\Gamma\circ u:[0,+\infty)\to\mathbb R$ is bounded.
	\begin{thrm}\label{thm:sparce_inversion}
		There exists a unique locally absolutely continuous function $v_{\varepsilon}:[0,+\infty)\to \mathcal H$ satisfying $v_{\varepsilon}(0)=u_0$ and (\ref{inversionmb1})-(\ref{inversionmb2}). Moreover, $v_{\varepsilon}$ is locally Lipschitz, and the following statements hold. 
		\begin{itemize}
			\item[(i)] For every $T>0$ there exists $c_{T}>0$ such that 
			\begin{align*}
				\sup_{t\in[0,T]}\mathbb E\big\| v_{\varepsilon}(t)-u(t)\big\|_{\mathcal H}^2\le c_{T}\varepsilon \int_{0}^T\Gamma\big(u(s)\big)\,ds \quad\forall \varepsilon>0.
			\end{align*}
			
			\item[(ii)] For every $\eta>0$ there exists $T>0$ such that 
			\begin{align*}
				\mathbb E\Phi(v_{\varepsilon}(T)) \le \inf_{v\in\mathcal H}\Phi(v) + \eta \quad \text{for all $\varepsilon>0$ small enough.}
			\end{align*}
			
			\item[(iii)] There exists $u^*\in\argmin_{v\in \mathbb R^d}\Phi(v)$ such that for every $\eta>0$ there exists $T>0$ such that
			\begin{align*}
				\mathbb E\big\|v_\varepsilon(T) - u^*\big\|_{\mathcal H}^2\le \eta \quad \text{for all $\varepsilon>0$ small enough.}
			\end{align*}
		\end{itemize}
	\end{thrm}
	\begin{proof}
		Define $\Phi_1:\mathbb R^d\to \mathbb R$ by  $\Phi_1(u):= (2\pi_1)^{-1}\big\|Au-b\big\|_2^2$ and $\Phi_2:\mathbb R^d\to\mathbb R$ by $\Phi_2(u):= \pi_2^{-1}\lambda\big\| u \big\|_{1}$.
		By Moreau–Rockafellar subdifferential additivity rule, $\partial \Phi =\pi_1\partial \Phi_{1} + \pi_2\partial\Phi_2$. For each $u\in \mathbb R^d$, denote by $\eta^*(u)$ the unique element in $\Theta(u)$ such that $\partial \Phi(u)^\circ = A^\top(Au-b)+ \eta^*(u)$. Let $\xi_{1}:\mathbb R^d\to \mathbb R$ be given by $\xi_1(u) = A^\top(Au-b)$, and $\xi_{2}:\mathbb R^d\to \mathbb R$ by $\xi_2(u) = \eta^*(u)$.
		We see that  $\pi_1\xi_j(u) + \pi_2 \xi_2(u) =\partial\Phi(u)^\circ$ for all $u\in \mathbb R^d$. Consider now, the function $\Lambda:\mathbb R^d\to \mathbb R$ in (\ref{Lambda}) based on the previous decomposition of the minimal norm subdifferential. It is not hard to see that  $	\Lambda(u)\le \Gamma(u)$ for all $u\in\mathbb R^d$. We can then employ Theorems \ref{existencethrm} and \ref{resultcontinuouscase} to conclude the result.
	\end{proof}

	\subsubsection{ Illustrative numerical example}
	To illustrate Theorem \ref{thm:sparce_inversion}, consider the following numerical example. Let $A\in \R^{2\times 2}$ and $b\in \R^2$ ($d=r=2$) given by
 \begin{align*}
     A=\begin{pmatrix}
         1.76 & 0.4\\
        0.98 & 2.24 
     \end{pmatrix},\quad b=\begin{pmatrix}
         1.87\\ -0.98
     \end{pmatrix},
 \end{align*}
 and we take $\lambda=1$. Let us denote by $u_*=(u_*^1,u_*^2)$ the optimal of the sparse inversion problem \eqref{sparseinversion}. In this particular case $u_*=(0.65,-0.45)$.  Now, let us consider the gradient flow \eqref{eq:gf_sparse}, associated with the sparse inversion problem in the fixed time interval $[0,T]$. We fix the time horizon $T=5$ and the initial condition $u_0=(0,0)$. For the implementation, we consider a time step $h=0.01$, and using an explicit Euler discretization, we numerically solve \eqref{eq:gf_sparse}. The trajectory computed is shown in \Cref{fig:sparse}. We can observe in  \Cref{fig:sparse} how the gradient flow converges to the minimum of the objective functional. The horizontal lines in \Cref{fig:sparse} (B) denote the optimal value $u_*=(u_*^1,u_*^2)$.  

\begin{figure}
\centering
\subfloat[]{
\includegraphics[height=4.7cm]{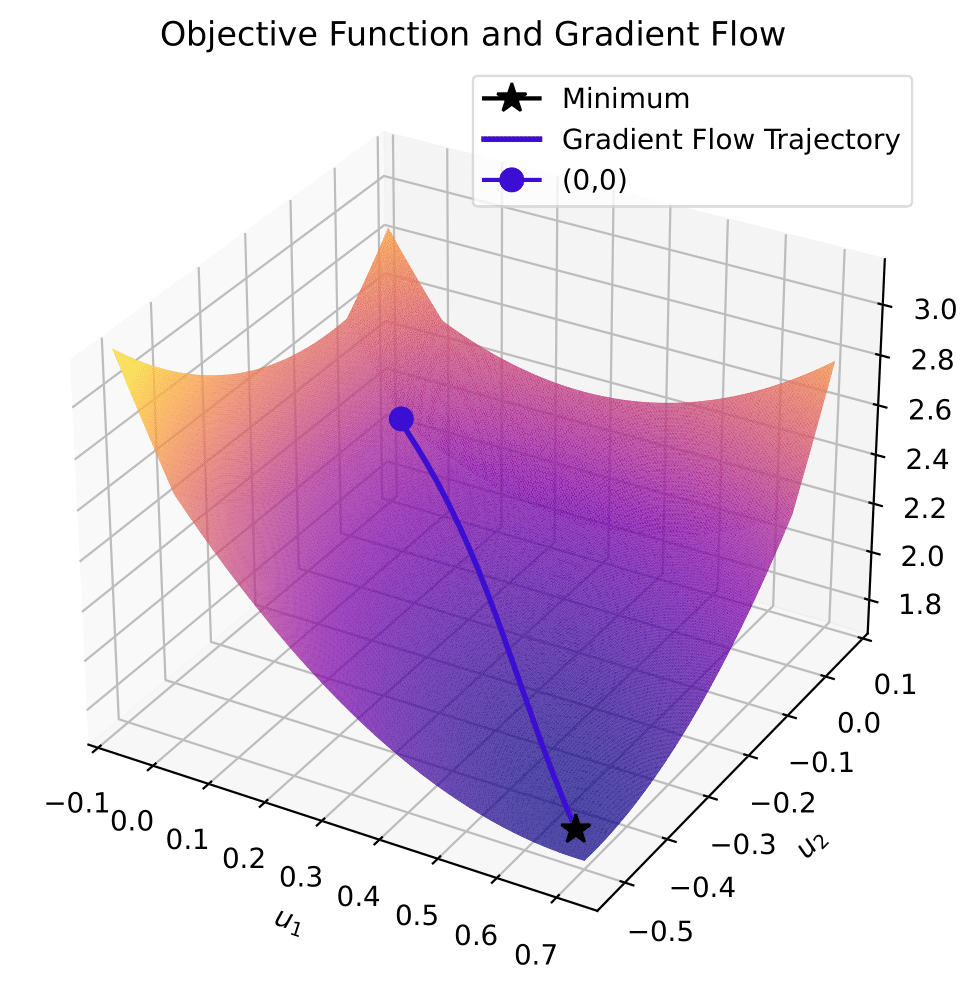}}
\subfloat[]{
\includegraphics[height=4.7cm]{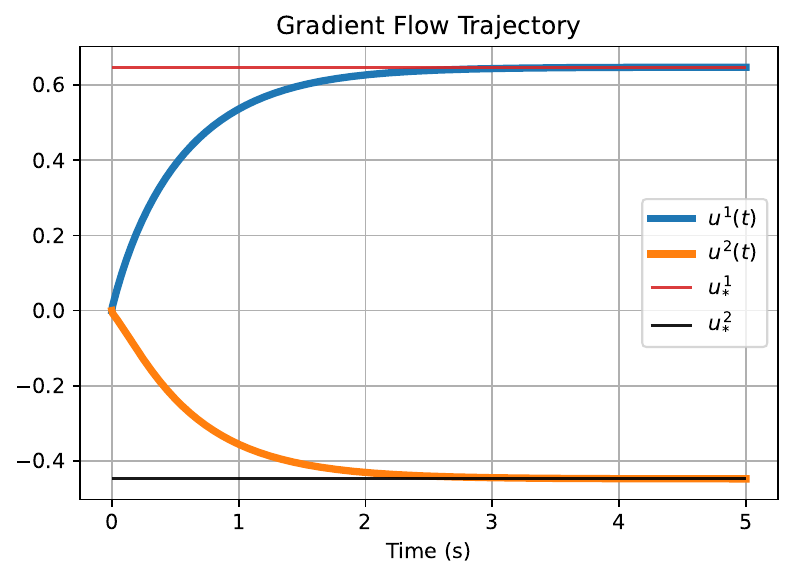}}
\caption{(A) Illustration of the gradient descent trajectories defined by system \eqref{eq:gf_sparse}. The black star denotes the minimum of $\Phi$. (B) Each curve corresponds to a different coordinate of the trajectory. Horizontal lines mark the optimal of $\Phi$.). }\label{fig:sparse}
\end{figure}
 \smallbreak

 On the other hand, for the mini-batch descent flow, we consider $\varepsilon=0.04$. Then, using the same setting of the gradient flow, we compute the solution of \eqref{inversionmb1}-\eqref{inversionmb2}. To approximate the expected trajectory of the mini-batch descent flow, we compute different realizations of the trajectory and then average them. In \Cref{fig:sparse2}, we observe different realizations of the mini-batch descent flow and their expected value.   

\begin{figure}
\centering
\subfloat[]{
\includegraphics[height=4.7cm]{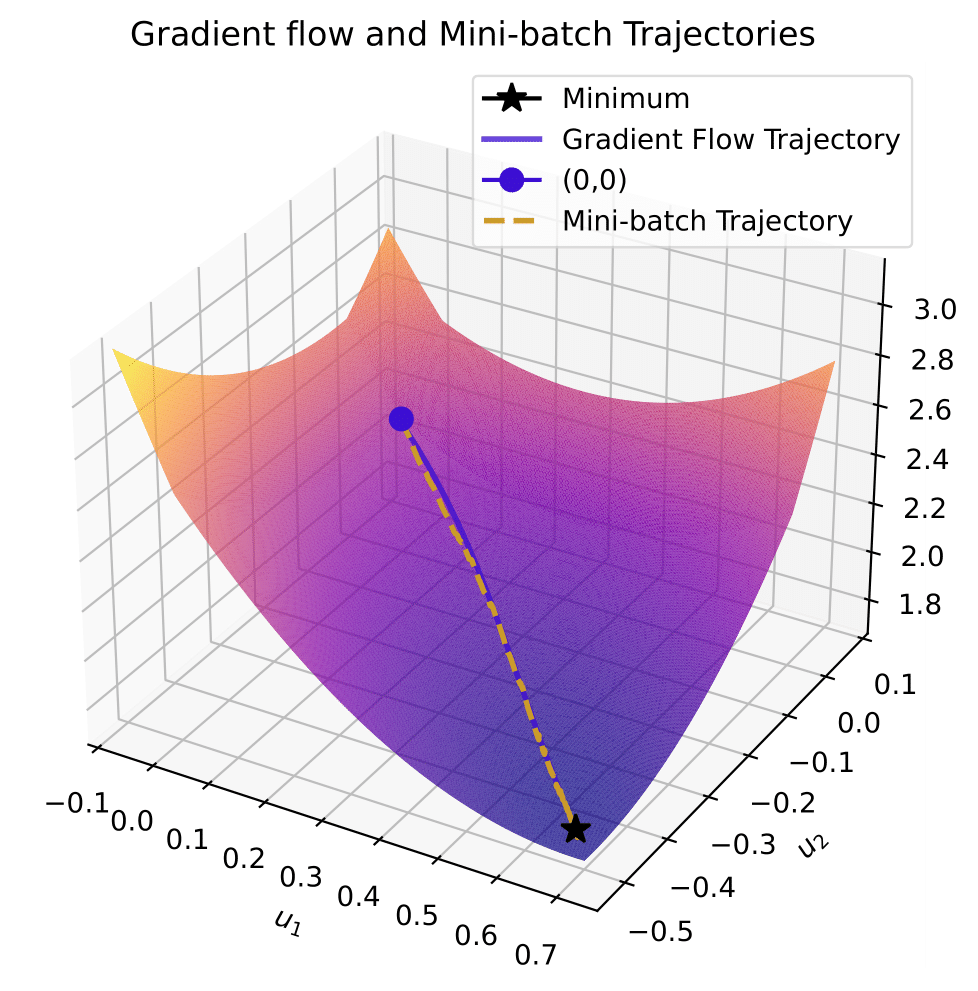}}
\subfloat[]{
\includegraphics[height=4.7cm]{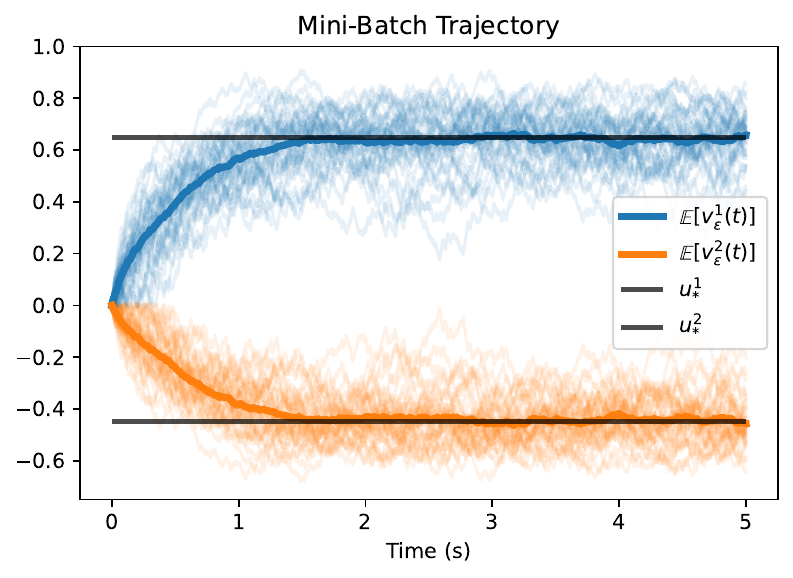}}
\caption{(A) Illustration of the expected value of the mini-batch descent flow trajectory. This trajectory star is from the same value of the gradient flow.
(B) Illustration of the mini-batch descent flow. Thin curves represent different realizations for each coordinate, while thick lines indicate the average outcomes of the mini-batch descent flow. }\label{fig:sparse2}
\end{figure}

To illustrate the rate of convergence guaranteed by Theorem \ref{thm:sparce_inversion}, we take $\varepsilon=1/K$ with $K$ being the number of batch switching in $[0,T]$. Then, \Cref{fig:sparse3} shows the convergence of the mini-batch descent flow to the gradient flow as $K\to \infty$ (equivalent $\varepsilon\to 0 $).

\begin{figure}
\centering
\includegraphics[width=0.4\textwidth]{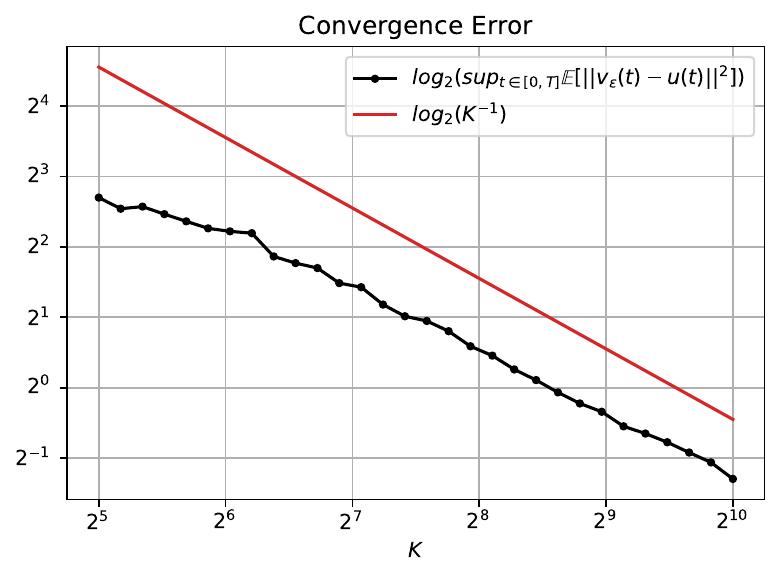}
\caption{Illustration of the convergence as $K \to \infty$ (equivalent to $\varepsilon \to 0$}\label{fig:sparse3}
\end{figure}

\smallbreak
In the previous example, we considered $d=r=2$. However, it is possible to consider a larger matrix size and compare the computation time taken to solve the gradient flow versus the mini-batch descent flow. This is illustrated in \Cref{tab:comp_time}, where the speedup corresponds to the ratio between the execution time of the gradient flow and the mini-batch descent flow.
\begin{table}
\centering
\caption{Computational Time Between Gradient and Mini-batch Algorithms. Here, we have considered $r=\lceil d/2\rceil$. The values of the matrices are chosen randomly from a uniform distribution $U(0,1)$.}
\label{tab:comp_time}
\begin{tabular}{ccccc}
\toprule
\textbf{Matrix Size A ($d \times d$)} & \textbf{Gradient Flow (s)} & \textbf{mini-batch descent flow (s)} & \textbf{Speedup} \\
\midrule
$5 \times 5$ & 0.005 & 0.004 & 1.289 \\
$50 \times 50$ & 0.006 & 0.005 & 1.248 \\
$100 \times 100$ & 0.008 & 0.006 & 1.316 \\
$200 \times 200$ & 0.017 & 0.014 & 1.236 \\
$400 \times 400$ & 0.040 & 0.033 & 1.206 \\
\bottomrule
\end{tabular}
\end{table}

 \Cref{tab:comp_time} shows that as the matrix size increases, the mini-batch descent flow  maintains a consistent speedup, demonstrating improved computational efficiency compared to the gradient flow.

	\subsection{Domain decomposition for the parabolic obstacle problem}\label{Domaindecom}
	In this section, we describe a random domain decomposition algorithm for parabolic-type equations that can be rewritten as gradient flows.

	In order to make the presentation simpler, we focus on the particular case of the obstacle problems.
	
	\subsubsection{Problem formulation}
	Let $\Omega\subset\mathbb R^d$ be an open bounded set. Let $\psi \in H^2(\Omega)$ be a given function satisfying $\psi\le0$ a.e. on $\partial \Omega$, from now on called \textit{the obstacle}, and let $K(\psi)$ be given by
	\begin{align*}
		K(\psi)=\left\{  u \in L^2(\Omega) \mid u(x) \geq \psi(x)\,\,\,\,\, \text{for  a.e. } x \in \Omega \right\}.
	\end{align*}
	Note that $K \neq \emptyset$ because $\psi^+ = \max(\psi,0) \in K$. Moreover, observe that $K$ is a convex subset of $L^2(\Omega)$. Let $T>0$ and consider the problem of find $u(t)\in K(\psi)\cap H_0^2(\Omega)$ for a.e. $t\in (0,T)$ such that for all $\phi\in K(\psi)$ the variational inequality
	\begin{align}\label{eq:obstacle_problem}
		\begin{cases}
			\int_{\Omega} \dot u (\phi-u) dx -\int_\Omega \Delta u (\phi-u) - \int_\Omega f(\phi-u) dx\geq 0\quad&
			\text{a.e. }t\in (0,T),\\
			u(x,0)=u_0(x)& \text{in }\Omega,
		\end{cases}
	\end{align}
	holds. This problem is known as the parabolic obstacle problem; in this case, we are considering a stationary obstacle. Concerning the well-posedness, we have the following result.

\begin{thrm}\label{th:regularity_u_obs}
		Let $u_0 \in K(\psi)$ and $f\in L^2(0,T;L^2(\Omega))$. Then, there exists a unique absolute continuous solution $u\in C([0,T];L^2(\Omega))$ of \eqref{eq:obstacle_problem}. Moreover, $u(t)\in \{\phi\in K(\psi)\cap H_0^1(\Omega)\,:\, \Delta \phi\in L^2(\Omega)\}$ for a.e. $t\in (0,T)$.
\end{thrm}
\begin{proof}
    Let us consider $\Phi:L^2(\Omega)\to\R$ defined as $\Phi = \Psi + \delta_{K(\psi)}$, where $\delta_{K(\psi)}$ is the indicator function of the closed convex $K(\psi)$, and 
    \begin{align*}
			\Psi(u):=\left\{ \begin{array}{lc} \displaystyle\frac{1}{2}\int_{\Omega} |\nabla u(x)|^2\,dx-\int_{\Omega}f(x)u(x)\, dx &\,\, if\,  u\in H_0^1(\Omega), \\ \\ +\infty &\,\, if\,  u\in L^2(\Omega),\, u\notin H_0^1(\Omega). \end{array} \right.
	\end{align*}
  Since $\Psi$ is a convex, lower semicontinuous, and proper function, it defines a maximal monotone operator $\partial\Psi$. On the other hand, we have that $\partial \Psi + N_{K(\psi)}\subset \partial(\Psi+\delta_{K(\psi)})=\partial\Phi$, where $ N_{K(\psi)}$ denote the normal cone of the set $K(\psi)$. Moreover, since Minty's Theorem \cite[Thrm 17.2.1]{Bottou2010} $\partial \Psi + N_{K(\psi)}$ is a maximal monotone operator, if only if, $R(I+\partial \Psi + \partial\delta_{K(\psi)})=L^2(\Omega)$ which is equivalent to the existence of solution of the equation $u -\Delta u + N_{K(\psi)}(u)\ni f$, with $u=0$ on $\partial\Omega$. The well-posedness is ensured by \cite[Proposition 2.11]{Barbu2010}. Consequently, by maximality, $\partial \Psi + N_{K(\psi)}=\partial\Phi$. Observe that $\text{dom }\partial \Phi = \{u \in K(\phi)\cap H_0^1(\Omega)\,:\, \Delta u \in L^2(\Omega)\}$.

  Now let us observe that the gradient flow associated with $\Phi$ is given by 
\begin{align}\label{eq:obstacle_wt_cone}
			\begin{cases}
				0\in \dot u-\Delta u -f + N_{K(\psi)} \quad &\text{in } \Omega\times (0,T),\\
				u(x,t)=0 & \text{on }\partial\Omega \times (0,T),\\
				u(x,0)=u_0(x)& \text{in }\Omega.
			\end{cases}
		\end{align}
	Using the definition of the normal cone mapping, observe that the variational inequality \eqref{eq:obstacle_problem} is equivalent to \eqref{eq:obstacle_wt_cone}. Therefore, due \cite[Theorem 17.2.5]{SoBV} there exists a unique strong solution of \eqref{eq:obstacle_problem} in $C([0,T];L^2(\Omega))$ such that $u(t)\in \text{dom }\partial \Phi$ a.e. $t\in (0,T)$.
  \end{proof}

	

For simplicity and to avoid formulating additional assumptions, we will assume that the solution of problem (\ref{eq:obstacle_problem}) belongs to $C^1([0,+\infty);L^2(\Omega))$. It is known from classic results that $\max\{\dot u,0\}$ belongs to $C^{1,1}([0,T]\times \bar \Omega)$, see \cite[Section 4]{Petrosyan} for further regularity results. If $\dot u$ is nonnegative, its continuity follows; see, e.g., \cite{Blanchet}. The nonnegativity of the time derivative can be established in some special cases (special initial conditions, boundary conditions, and time-independent coefficients); see, for example, \cite{Blanchet, Friedman}.

	\subsubsection{Random domain decomposition and convergence}
	
	In order to decompose the domain, let $n\in \mathbb{N}$ and consider a non-overlapping partition $\{\Omega_{i}\}_{i=1}^n$ of $\Omega$. Let us introduce $\{\chi_{\Omega_i}\}_{i=1}^n\subset W^{1,\infty}(\Omega)$ a partition of unity subordinate to $\{\Omega_{i}\}_{i=1}^n$, that is,
 \begin{align}
\sum_{i=1}^n\chi_{\Omega_i} =1\quad \text{in }\Omega,\quad\text{and}\quad \text{supp} (\chi_{\Omega_i})\subset \Omega_i, \quad \text{for every}\, i\in\{1,\dots, n\}.
 \end{align}
In a completely analogous way to \Cref{sec:problem_formulation} we introduce the batches $\{B_j\}_{j=1}^m$ such that $\cup_{j=1}^m B_j=\{1,\dots, n\}$,  the probabilities $\{p_i\}_{i=1}^n$ and  $\{\pi_i\}_{i=1}^n$, and the sequence of independent random variables $\{j_k\}_{k\in\mathbb{N}}$. Therefore, let us consider the function $w_{\varepsilon}:[0,+\infty)\to L^2(\Omega)$ solution of
	
	\begin{align}\label{eq:mbdgpde}
		\begin{cases}
			\int_{\Omega} \left(\frac{w_{k}-w_{k-1}}{\varepsilon}\right) (\phi-w_k) dx + A(w_k,\phi,j_k) - F(\phi,j_k)\geq 0\,\,&
			\text{a.e. }k\in\mathbb{N},\\
			w_0(x)=u_0(x)& \text{in }\Omega.
		\end{cases}
	\end{align}
where $A(w_k,\phi,j_k)$ and $ F(\phi,j_k)$ are given by 
\begin{align}\label{eq:A_rmm_domain}
    A(w_k,\phi,j_k)=\frac{1}{|B_{j_{k_t}}|}\sum_{i\in B_{j_{k_t}}}\int_{\Omega}\dive(\chi_{\Omega_i} \nabla w_k )(\phi-w_k),
\end{align}
and 
\begin{align}\label{eq:F_rmm_domain}
     F(\phi,j_k)= \frac{1}{|B_{j_{k_t}}|}\sum_{i\in B_{j_{k_t}}}\int_{\Omega}\chi_{\Omega_i} f(\phi-w_k) dx.
\end{align}
To embed sequence $\{w_k\}$ into $L_{loc}^2([0,+\infty);L^2(\Omega))$, we consider the function $w:[0,+\infty)\to L^2(\Omega)$ given by $w(t):=w_{k}$ if $t\in[t_{k-1},t_k)$.
 The previous system can be understood as a randomization of both the principal operator and the source. It is important to note that for a $k\in \mathbb{N}$ in which the subdomain $\Omega_*$ has not been selected, the solution $w_k$ of the \eqref{eq:mbdgpde} remains constant for the next step; that is, $w_{k+1}=w_k$  on $\Omega_*$. We have the following theorem. In the following, we assume that $\dot u$ is continuous then the following result holds.

	\begin{thrm}\label{theorem_obstacle_mbm}
		Let $u:[0,+\infty)\to L^2(\Omega)$ be the solution of the obstacle problem (\ref{eq:obstacle_problem}) and $w_{\varepsilon}:[0,+\infty)\to L^2(\Omega)$ the solution of (\ref{eq:mbdgpde}). Then, for each $t\in[0,+\infty)$,
		\begin{align}\label{eq:convergence_obstacle}
			\mathbb E\big\|w_\varepsilon(t) - u(t)\big\|_{\mathcal H}^2\longrightarrow 0\quad\text{as}\quad \varepsilon\longrightarrow0^+.
		\end{align}
		Moreover, let $u^*\in H_0^1(\Omega)$ be the solution of the (stationary) obstacle problem
		\begin{align*}
  \int_\Omega \nabla u^* \nabla(\phi-u^*) - \int_\Omega f(\phi-u^*) dx\geq 0.
		\end{align*}
		Then, for every $\eta>0$ there exists $T>0$ such that
		\begin{align}\label{eq:mm_behivior}
			\mathbb E\big\|w_\varepsilon(T) - u^*\big\|_{\mathcal H}^2\le \eta \quad \text{for all $\varepsilon>0$ small enough.}
		\end{align}
	\end{thrm}
	
			
	\begin{proof}
Let us consider $\Phi$ the functional defined in the proof of Theorem \ref{th:regularity_u_obs}. Then, the gradient flow associated with $\Phi$ is given by \eqref{eq:obstacle_wt_cone}.
		Using the definition of the normal cone mapping, observe that the variational inequality \eqref{eq:obstacle_problem} is equivalent to \eqref{eq:obstacle_wt_cone}. On the other hand, for each $i\in\{1,\dots, n\}$, we introduce  
		\begin{align*}
			\Psi_i(u):=\left\{ \begin{array}{lc} \displaystyle\frac{1}{2}\int_{\Omega} \chi_{\Omega_i}(x)|\nabla u(x)|^2\, dx-\int_{\Omega}\chi_{\Omega_i}f(x)u(x)\, dx &\,\, if\,   u\in H_0^1(\Omega), \\ \\ +\infty &\,\,  if\,  u\in L^2(\Omega),\, u\notin H_0^1(\Omega). \end{array} \right.
		\end{align*}
Let us define the functional $\Phi_i:L^2(\Omega)\to\R$ given by $\Phi_i= \Psi_i + \delta_{K(\psi)}$. Analogous to Theorem \ref{th:regularity_u_obs}, we can deduce that $\partial\Phi= \partial \Psi +N_{K(\psi)}$. 
Observe that $\dom\partial \Phi_{i} =\{u\in K(\psi)\cap H_0^1(\Omega):\, \dive\big(\chi_{\Omega_i}\nabla u\big)\in L^2(\Omega)\}$ and $\partial\Phi_i(u) = -\dive\big(\chi_{\Omega_i}\nabla u \big) -\chi_{\Omega_i}f + N_{K}(u)$. Therefore, in this case, the random minimizing movement introduced in \Cref{sec:RMM} is given by the solution of 
\begin{align}\label{mbdgpde_2}
			\displaystyle-\frac{w_{k}-w_{k-1}}{\varepsilon} + A_{j_{k}}(w_k)\in N_{K(\psi)}(w_k),\quad \forall k\in\mathbb N.
		\end{align}
		where $A_{j_{k_t}}: H_0^1(\Omega) \subset L^2(\Omega)\to L^2(\Omega)$ is defined by
		\begin{align*}
			A_{j_{k_t}}(v_\varepsilon)=\frac{1}{\mathcal |B_{j_{k_t}}|}\sum_{i\in\mathcal B_{j_{k_t}}}\big(\dive\left( \chi_{\Omega_i}\nabla v_\varepsilon\right) -\chi_{\Omega_i}f \big) ,\quad \text{for every }k\in \mathbb{N}.
		\end{align*}
		
		Observe that $\Lambda\circ u:[0,+\infty)\to L^2(\Omega)$ corresponde to the norm $H^2_0(\Omega)$ of $u$ and is locally bounded due Theorem \ref{th:regularity_u_obs}. Therefore, since Theorem \ref{th:conv_mm} we deduce \eqref{eq:convergence_obstacle}. Moreover, as a consequence of Theorem \ref{th:asym_beh}, we deduce that for every $\eta>0$, there exists $T>0$ such that \eqref{eq:mm_behivior} holds. 
\end{proof}
\subsubsection{Illustrative numerical example}
In this section, we will illustrate Theorem \ref{theorem_obstacle_mbm} with a numerical example. For this purpose, let us consider $\Omega=[-1,1]\times[-1,1]\subset\R^2$ and the obstacle 
\begin{align*}
\psi(x)=\begin{cases}
    -4(x+0.5)^2-4y^2 &\text{ if }4(x+0.5)^2+4y^2<1,\\
    -4(x-0.5)^2-4y^2 &\text{ if }4(x-0.5)^2+4y^2<1, \\
    0&\text{ otherwise.}
\end{cases}    
\end{align*}
We take $u_0(x)=0$ as the initial condition and $f=-1$ as the source term. This source term acts as a gravity force, pushing $u$ to be close to the obstacle. For the numerical implementation of \eqref{eq:obstacle_problem}, for every $\delta>0$ we consider the penalized problem
\begin{align}\label{eq:penalized_problem_domain}
    \begin{cases}
        \int_{\Omega} \dot u^\delta \phi\,dx +\int_\Omega \Delta u^\delta  \phi\, dx - \frac{1}{\delta}\int_\Omega\max\{-u^\delta + \psi,0\}\phi\,dx  - \int_\Omega f\phi\,dx = 0\quad&
        \text{a.e. }t\in (0,T),\\
        u^\delta(x,t)=0 & \text{on }\partial\Omega \times (0,T),\\
        u^\delta(x,0)=u_0(x)& \text{in }\Omega,
		\end{cases}
\end{align}
As is shown in \cite[Proposition 2.2.]{MR1896217} $u^\delta \to u$ in $L^2(0,T;L^2(\Omega))$ as $\delta\to 0$. On the other hand, we consider the following subdomains 
\begin{equation*}
\begin{aligned}
     \Omega_1=&[-1,0.1]\times [-1,0.1], \quad \Omega_2 = [-0.1 , 1]\times [-1,0.1],\\
    \Omega_3 = [-0&.1,1]\times[-1,0.1],\quad \text{
    and
    } \quad\Omega_4=[-0.1,1]\times [-0.1,1],
\end{aligned}
\end{equation*}
 and the function 
\begin{align*}
    h(x)=\begin{cases}
        0 &\text{ if } x\geq -0.1\\
        \frac{r+0.1}{0.2} &\text{ if } x\in[-0.1,0.1],\\
        1 &\text{ otherwise}.
    \end{cases}
\end{align*}
The partition is then given by the functions 
\begin{equation}\label{eq:definition_chi_obstacle}
    \begin{aligned}
        \chi_{\Omega_1}(x,y) &= (1-h(x))(1-h(y)), \quad 
        \chi_{\Omega_2}(x,y) = h(x)(1-h(y)), \\
        \chi_{\Omega_3}(x,&y) = (1-h(x))h(y), \quad \text{and}\quad
        \chi_{\Omega_4}(x,y) = h(x)h(y).
    \end{aligned}
\end{equation}
These functions are illustrated \Cref{fig:chi_functions}.
 \begin{figure}
\centering
\includegraphics[width=0.9\textwidth]{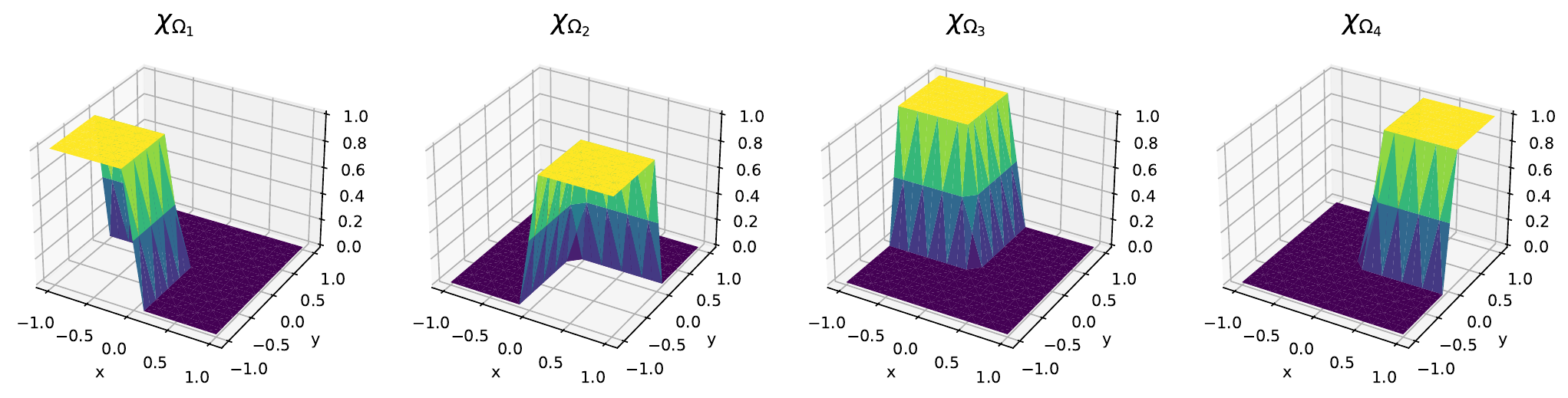}
\caption{Illustration of funtions $\chi_{\Omega_i}$. Observe that $\{\chi_i\}_{i=1}^4$ is a partition of unity suborditnate to $\{\Omega_i\}_{i=1}^4$. }\label{fig:chi_functions}
\end{figure}
 We can then condider batches $B_i={i}$ and its respective probability for each $i\in\{1,\dots,4\}$. Then, we introduce a random minimizing movement scheme of \eqref{eq:penalized_problem_domain} as
\begin{align}\label{eq:mbdgpde_pena}
		\begin{cases}
			\int_{\Omega} \left(\frac{w_{k}^\delta-w_{k-1}^\delta}{\varepsilon}\right) (\phi-w_k^\delta) dx + A(w_k^\delta,\phi,j_k) - \frac{1}{\delta}\int_\Omega\max\{-w^\delta + \psi,0\}\phi\,dx - F(\phi,j_k)= 0\,\,&
			\text{a.e. }k\in\mathbb{N},\\
			w_0^\delta(x)=u_0(x)& \text{in }\Omega,
		\end{cases}
	\end{align}
  where $A$ and $F$ are defined as in \eqref{eq:A_rmm_domain} and \eqref{eq:F_rmm_domain} respectively.

We take $T=0.5$ as the time horizon. For the implementation of \eqref{eq:penalized_problem_domain}, we consider an implicit Euler discretization of the time interval, using $60$ discretization points. For space discretization, we use finite elements, with basis functions given by the standard continuous Galerkin functions (piecewise polynomial functions of order one). Also, the penalization parameter $\delta=10^{-8}$ and the $\max$ function have been regularized using a smooth max function (see \cite{MR2822818} for a rigorous analysis). We have used $20$ elements. The implementation was carried out using FEniCS, an open-source computing platform finite elements.

For the implementation of \eqref{eq:mbdgpde_pena}, we used FEniCS to compute the solution. To approximate the average of the solution, we used $8$ realizations and then averaged them. The result can be seen in \Cref{fig:rand_vs_det_dd}.

\begin{figure}
\centering
\includegraphics[width=0.9\textwidth]{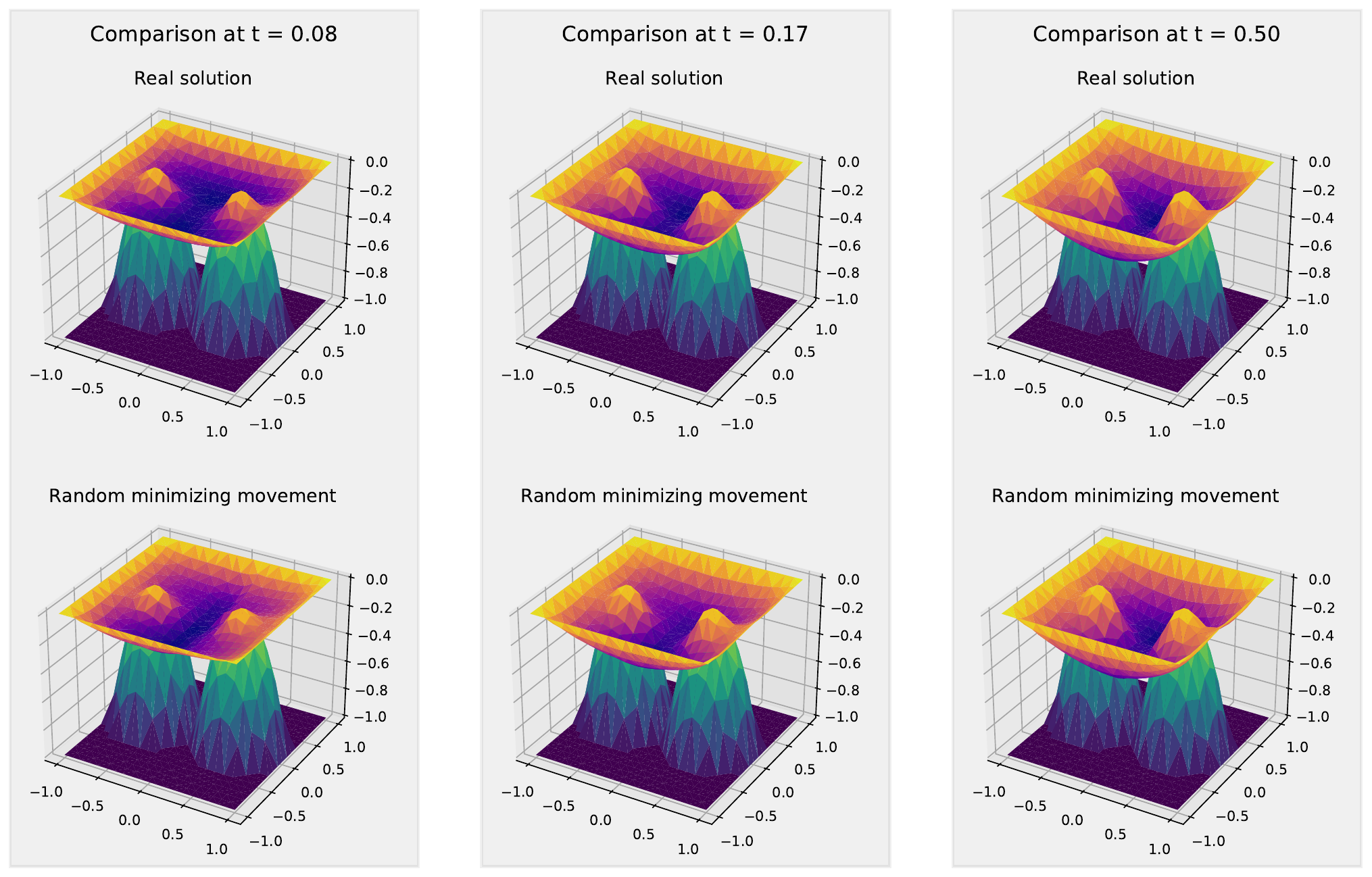}
\caption{Comparison between the obstacle and random minimizing movement solutions (i.e., systems \eqref{eq:penalized_problem_domain} and \eqref{eq:mbdgpde_pena}) for different times.}\label{fig:rand_vs_det_dd}
\end{figure}

To illustrate the convergence of the \textit{Random domain decomposition}, we use $8$ realizations to approximate the average, and $25$ elements fixed for space discretization. Denoting by $K$ the number of time steps that we are considering ($\varepsilon=1/K$ in \eqref{eq:mbdgpde_pena}), \Cref{fig:convg_obstacle} illustrates the convergence of this scheme. Observe that we can guarantee convergence, but there has been no convergence rate since our hypothesis. However, the numerical evidence suggests that the convergence order of \textit{Random domain decomposition} is $O(\varepsilon)$.

\begin{figure}
\centering
\includegraphics[width=0.4\textwidth]{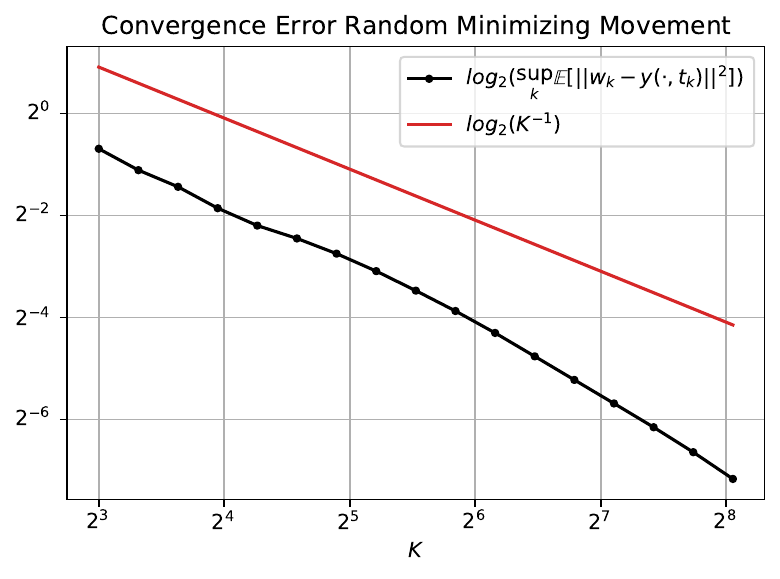}
\caption{We illustrate the convergence as $K\to \infty$ (equivalent to $\varepsilon\to0$) of the random minimizing movement.}\label{fig:convg_obstacle}
\end{figure}






\appendix
\section{Appendix: Variance of mini-batch descent}\label{app:variance}

In the following, we present some commentaries concerning the variance of mini-batch descent. 
	Let $\Phi:\mathcal H\to\mathbb R\cup\{+\infty\}$ be the convex proper lower semicontinuous functional in Section (\ref{sec:problem_formulation}). Due to the definition of probabilities $\pi_1,\dots, \pi_m$ it is possible to see that 
	\begin{align}\label{nocheinever}
		\Phi = \sum_{j=1}^{m}\pi_j\Phi_{\mathcal B_j}.
	\end{align}
	This, in turn, says that $\Phi$ can be seen as the expected value of a random variable taking value $\Phi_{\mathcal B_j}$ with probability $\pi_j$.   Consider the sequence of independent random variables $\{j_k\}_{k\in\mathbb N}$ in Section \ref{sec:problem_formulation}. We see from (\ref{nocheinever}) that $\mathbb E \Phi_{\mathcal B_{j_k}} = \Phi$ for all $k\in\mathbb N$. 
	
	In order for the mini-batch descent to be a non-biased estimator, in the sense that the expected value of the randomly chosen subgradient coincides with the subdifferential of $\Phi$ at each step, the sum rule for subdifferentials  was assumed to hold in Section \ref{sec:problem_formulation}, i.e., 
	\begin{align}\label{MRsub}
		\partial \Phi(u) = \sum_{j=1}^m \pi_{j}\, \partial \Phi_{\mathcal B_{j}}(u)\quad \text{for all $u\in\dom\partial\Phi$.}
	\end{align}
	This rule can be thought of as a structural assumption allowing the set-valued variable taking value $\partial \Phi_{\mathcal B_j}$ with probability $\pi_j$ can have expected value $\partial \Phi(u)$. For differentiable potentials, this holds automatically since the gradient operator is linear.
	For the validity of (\ref{nocheinever}), we refer to the general result \cite[Corollary 16.50]{BauCom}; see also \cite[Theorem 3.1]{Bura} for a sufficient dual condition.
	\medbreak
 
	{\bf 1. Splitting of the minimal norm subdifferential:}
	Additivity assumption (\ref{MRsub}) ensures that for each $u\in\dom\Phi$ there  exist $\xi_1(u),\dots,\xi_m(u)\in\mathcal H$ such that $\xi(u)\in \partial \Phi_{\mathcal B_j}(u)$ for $j\in\{1,\dots,m\}$, and $\partial \Phi(u)^\circ = \sum_{j=1}^{m} \pi_j\xi_j(u)$.
	In sparse inversion problem we presented the splitting of the minimal norm subdifferential was unique. 
	However, in the general case, without further assumption on mini-batch potentials, there is no reason for this decomposition to be unique; thus, it should be chosen according to the problem at hand.  From now on, we assume that functions $\xi_{1},\dots,\xi_m:\dom\Phi:\to\mathcal H$ are fixed.

	It can be readily seen that for each $k\in\mathbb N$, $\mathbb E \xi_{j_k}(u) = \partial \Phi(u)^\circ$ for all $u\in\dom\Phi$.

	In Section \ref{sec:problem_formulation}, we introduced function $\Lambda:\mathcal H\to \mathbb R$ given by $\Lambda(u):= {\sum_{j=1}^m \pi_j|\xi_{j}(u) -\partial \Phi(u)^\circ|^2_{\mathcal H}}$ to provide a quantitative measure of variance for mini-batch descent.  When reduced to the case where mini-batch potentials are differentiable, $\Lambda$ is the usual quantifier of variance used to provide bounds and estimates in stochastic gradient algorithms. 
	The key feature of $\Lambda$ is that $\Var\left[ \xi_{j_k}\right]=\Lambda$ for all $k\in\mathbb N$. 
	\medbreak

	{\bf 2. Local boundedness of $\Lambda$:}
	Let $u:[0,+\infty)\to\mathcal H$ be the solution of gradient flow equation (\ref{gf_eq}). All hypotheses made in Section \ref{sec:problem_formulation} require at least local integrability of $\Lambda\circ u:[0,+\infty)\to\mathbb R$. In all the examples we provided (sparse inversion, constrained optimization, and domain decomposition), this function was locally bounded. 
	
	In general, there is a criterion to assess the local boundedness of the subdifferential; see \cite{BauCom}. This can be employed to give a sufficient condition under which  $\Lambda$ is locally bounded (this is stronger than $\Lambda\circ u$ being locally bounded). For any $v\in \mathcal H$, 
	\begin{align*}
		\text{if $v\in \inte \dom\Phi_{\mathcal B_{j}}$  $\forall j\in\{1,\dots, m\}$, then $\exists \mathcal V\in \mathcal N(v)$ such that  $\Lambda(\mathcal V)$ is bounded.}
	\end{align*}
	Here, $\mathcal N(v)$ denotes the neighborhood filter of $v\in\mathcal H$. Therefore, the openness of the effective domains of mini-batch potentials is, in general, enough to grant the local boundedness of $\Lambda\circ u$. However, in general, and due to the properties of $u$, it is much easier to bound $\Lambda\circ u$, as shown in the example in Section \ref{Sparseinv}, where $\Lambda\circ u$ is bounded all over $[0,+\infty)$.
	
	\medbreak

	{\bf 3. Another interpretation of function $\Lambda$:}
	For simplicity, suppose that $m$ divides $n$, and that batches contain exactly $\frac{n}{m}$ elements. Moreover, assume that they are selected with equal probability, i.e., $\pi_j = m^{-1}$ for all $j\in\{1,\dots,m\}$.
	
	In this case,  the potential can be represented as 
	\begin{align*}
		\Phi(u) = \frac{1}{n}\sum_{j=1}^m\sum_{i\in\mathcal B_j} \Phi_i(u)\quad \forall u\in\dom\Phi.
	\end{align*}
	Assuming that $\partial \Phi_{\mathcal B_j} = m^{-1}\sum_{i\in\mathcal B_j}\partial \Phi_i$, it is possible to find functions $\xi_{i,j}:\dom\Phi\to\dom\Phi_i$ such that $\partial \Phi(u)^\circ = n^{-1}\sum_{j=1}^{m}\sum_{i\in\mathcal B_j}\xi_{i,j}(u)$ for all $u\in\dom\Phi$. Consider the function $\Gamma:
	\mathcal H\to \mathbb R$ given by 
	\begin{align*}
		\Gamma(u):=\frac{1}{n}\sum_{j=1}^m\sum_{i\in\mathcal B_j} \big\|\xi_{i,j}(u) - \partial \Phi(u)^\circ\big\|_{\mathcal H}^2
	\end{align*}
	This function can be readily identified as the trace of the empirical covariance matrix of per-example subgradients. 
	
	Following \cite[Appendix A]{Smith2020origin}, it is possible to find that 
	\begin{align*}
		\Lambda(u) = \frac{n-m}{m(n-1)}\Gamma(u)\quad \forall u\in\dom \Phi.
	\end{align*}
	

\vspace{5mm}
\section*{Acknowledgments}

  A. Dom\'{i}nguez Corella is supported by the Alexander von Humboldt Foundation with an Alexander von Humboldt research fellowship and by the Emerging Talents Initiative (FAUeti) funding. M. Hern\'{a}ndez has been funded by the Transregio 154 Project, Mathematical Modelling, Simulation, and Optimization Using the Example of Gas Networks of the DFG, project C07, and the fellowship "ANID-DAAD bilateral agreement". Both authors have been partially supported by the DAAD/CAPES Programs for Project-Related Personal, grant 57703041 'Control and numerical analysis of complex system'.

The authors wish to express their gratitude to Enrique Zuazua  for insightful discussions, and for his constant and generous support.

\bibliography{references.bib}
\bibliographystyle{abbrv}
\end{document}